\documentclass[11pt, a4paper]{amsart}
\usepackage[utf8]{inputenc}
\usepackage{amsmath}
\usepackage{amssymb}
\usepackage{amsfonts}
\usepackage{amscd}
\usepackage{enumerate}
\usepackage[margin = 0.9in]{geometry}
\usepackage{amsthm}
\usepackage{mathrsfs}
\usepackage{mathtools}
\usepackage{amssymb}
\usepackage[all]{xy}
\usepackage{xcolor}
\usepackage{graphics}
\usepackage{lscape}
\usepackage{array}
\usepackage{microtype}
\usepackage{setspace}
\usepackage{adjustbox}
\usepackage{palatino}
\usepackage{eulervm}
\usepackage{parskip}
\usepackage{marvosym}
\usepackage{tikz-cd} 
\usetikzlibrary{graphs,decorations.pathmorphing,decorations.markings}
\usepackage{stmaryrd} 
\usepackage{upgreek} 
\usepackage{xcolor}
\usepackage{centernot} 
\usepackage[utf8]{inputenc}
\usepackage{comment}
\usepackage[shortlabels]{enumitem}
\usepackage{xy}

\makeatletter

\def\l@section{\@tocline{1}{0pt}{0pt}{2.3em}{}}
\def\l@subsection{\@tocline{2}{0pt}{2em}{3em}{}}
\def\l@subsubsection{\@tocline{3}{0pt}{4em}{4em}{}}

\makeatother

\numberwithin{equation}{section}
\newtheorem*{theorem*}{Theorem}
\newtheorem*{definition*}{Definition}
\newtheorem*{theorem_A}{Theorem A}
\newtheorem*{theorem_B}{Theorem B}

\newtheorem*{conjecture_A}{Conjecture A}

\newtheorem{theorem}{Theorem}[section]
\newtheorem{lemma}[theorem]{Lemma}
\newtheorem{proposition}[theorem]{Proposition}
\newtheorem{corollary}[theorem]{Corollary}
\newtheorem{conjecture}[theorem]{Conjecture}
\newtheorem{remark}[theorem]{Remark}
\theoremstyle{definition}
\newtheorem{definition}[theorem]{Definition}
\newtheorem{def/prop}[theorem]{Definition/Proposition}

\newtheorem{example}[theorem]{Example}
\theoremstyle{remark}

\usepackage{hyperref}\hypersetup{colorlinks}

\usepackage{color} 

\definecolor{darkred}{rgb}{1,0,0} 
\definecolor{darkgreen}{rgb}{0,1,0}
\definecolor{darkblue}{rgb}{0, 0, 1}
\definecolor{darkpurple}{RGB}{170, 51, 106}

\hypersetup{colorlinks,
linkcolor=darkblue,
filecolor=darkgreen,
urlcolor=darkred,
citecolor=darkpurple}

\DeclareMathAlphabet\mathbfcal{OMS}{cmsy}{b}{n}

\DeclareMathOperator{\Hom}{Hom}

\DeclareMathOperator{\End}{End}

\DeclareMathOperator{\id}{1}
\DeclareMathOperator{\Spec}{Spec}

\DeclareMathOperator{\Coh}{Coh}

\DeclareMathOperator{\Higgs}{Higgs}

\DeclareMathOperator{\Mor}{Mor}
\DeclareMathOperator{\Maps}{Maps}
\DeclareMathOperator{\anMaps}{anMaps}
\DeclareMathOperator{\an}{an}

\DeclareMathOperator{\ad}{ad}

\DeclareMathOperator{\Loc}{Loc}

\DeclareMathOperator{\Sym}{Sym}

\DeclareMathOperator{\Rep}{Rep}

\DeclareMathOperator{\Ob}{Ob}

\DeclareMathOperator{\Hodge}{Hodge}
\DeclareMathOperator{\Deligne}{Del}
\DeclareMathOperator{\Tw}{Tw}

\DeclareMathOperator{\Betti}{Betti}

\DeclareMathOperator{\triv}{triv}
\DeclareMathOperator{\Triv}{Triv}

\DeclareMathOperator{\Perf}{Perf}

\DeclareMathOperator{\Dol}{Dol}
\DeclareMathOperator{\dR}{dR}
\DeclareMathOperator{\Hod}{Hod}
\DeclareMathOperator{\Sim}{Sim}
\DeclareMathOperator{\Del}{Del}
\renewcommand{\top}{\mathrm{top}}
\DeclareMathOperator{\QC}{QC}
\DeclareMathOperator{\QA}{QA}

\DeclareMathOperator{\BBB}{BBB}
\DeclareMathOperator{\sst}{sst}
\DeclareMathOperator{\st}{st}

\DeclareMathOperator{\im}{im}

\newcommand{\shear}{{\mathbin{\mkern-6mu\fatslash}}}
\DeclareRobustCommand{\unshear}{\text{\reflectbox{$\shear$}}}

\DeclareMathOperator{\qalg}{\mathrm{qalg}}
\DeclareMathOperator{\QAlg}{\mathrm{QAlg}}

\newcommand{\Hhom}{\mathcal{H}om}

\DeclareMathOperator{\Ind}{\mathrm{Ind}}
\DeclareMathOperator{\discrete}{\mathrm{disc}}

\newcommand{\Cc}{\mathcal{C}}

\newcommand{\Ee}{\mathcal{E}}
\newcommand{\Ff}{\mathcal{F}}

\newcommand{\Kk}{\mathcal{K}}
\newcommand{\Ll}{\mathcal{L}}
\newcommand{\Mm}{\mathcal{M}}

\newcommand{\Oo}{\mathcal{O}}

\newcommand{\Rr}{\mathcal{R}}
\newcommand{\Ss}{\mathcal{S}}
\newcommand{\Tt}{\mathcal{T}}
\newcommand{\Uu}{\mathcal{U}}

\newcommand{\Yy}{\mathcal{Y}}
\newcommand{\Zz}{\mathcal{Z}}

\newcommand{\Ggr}{\mathbb{G}_{\mathrm{gr}}}

\newcommand{\bB}{\mathbf{B}}
\newcommand{\C}{\mathbf{C}}

\renewcommand{\H}{\mathbf{H}}

\renewcommand{\O}{\mathrm{O}}
\newcommand{\OO}{\mathbf{O}}

\newcommand{\B}{\mathrm{B}}
\renewcommand{\P}{\mathbf{P}}

\newcommand{\g}{\mathrm{g}}

\newcommand{\ol}[1]{\overline{#1}}

\newcommand{\wt}[1]{\widetilde{#1}}

\renewcommand{\AA}{\mathbb{A}}

\newcommand{\CC}{\mathbb{C}}
\newcommand{\DD}{\mathbb{D}}

\newcommand{\GG}{\mathbb{G}}
\newcommand{\RR}{\mathbb{R}}
\newcommand{\ZZ}{\mathbb{Z}}
\newcommand{\LL}{\mathbb{L}}

\newcommand{\PP}{\mathbb{P}}

\newcommand{\TT}{\mathbb{T}}

\renewcommand{\to}{\longrightarrow}

\newcommand{\acts}{\curvearrowright}
\newcommand{\racts}{\curvearrowleft}

\newcommand\Quotient[2]{
\mathchoice
{
\text{\raise1ex\hbox{\thinspace $#1$}\Big{/} \lower1ex\hbox{$#2$} \thinspace}%
}
{
#1\,/\,#2
}
{
#1\,/\,#2
}
{
#1\,/\,#2
}
}

\newcommand\GIT[2]{
\mathchoice
{
\text{\raise1ex\hbox{\thinspace $#1$}\Big{/}\!\!\!\!\Big{/} \lower1ex\hbox{$#2$} \thinspace}%
}
{
#1\,/\,#2
}
{
#1\,/\,#2
}
{
#1\,/\,#2
a       }
}

\thanks{
E.Y.C. would like to recognize the support of the Swiss National Science Foundation No. 196960 and the JSPS Postdoctoral Fellowship during the completion of this project.
E.F. was partially supported by the Spanish Ministry of Science and Innovation, through the ‘Severo Ochoa Programme for Centres of Excellence in R$\&$D’ (CEX2019-000904-S), and through projects PID2022-141387NB-C22 and PID2025-174260NB-C21.}

\begin{document}

\author[E. Chen]{Eric Yen-Yo Chen}
\address{E. Y. Chen, \newline\indent \'Ecole Polytechnique F\'ed\'erale de Lausanne, 
\newline\indent CH-1015 Lausanne, Switzerland.}
\email{eric.chen@epfl.ch}

\author[E. Franco]{Emilio Franco}
\address{E. Franco,
\newline\indent Depto. Matem\'aticas, Facultad de Ciencias, 
\newline\indent Universidad Aut\'onoma de Madrid
\newline\indent Campus de Cantoblanco 28049, Madrid, Espa\~na.}
\email{emilio.franco@uam.es}

\title{Quasi-algebraic quantization for the B-twist Langlands TQFT}

\begin{abstract}
This is the first part of a program to construct hyperholomorphic families of boundary conditions for the Kapustin--Witten B-twist of the Langlands QFT, otherwise known as \textit{(BBB)-branes}. We define the category of quasi-algebraic sheaves over the Deligne moduli stack, which serves as an analog of the twistor space of Hitchin's moduli stack. This allow us to construct a representation of a simplified version of the Moore--Tachikawa category which is motivated by the relative Langlands program in the sense of Ben-Zvi--Sakellaridis--Venkatesh.
\end{abstract}

\maketitle

\begingroup
\hypersetup{linkcolor=black}
\tableofcontents
\endgroup

\section{Introduction}

\subsubsection*{Background and motivation}

S-duality, as envisioned by Kapustin--Witten \cite{kapustin&witten} and Gaiotto--Witten \cite{gaiotto&witten}, suggests that Langlands duality should be understood as a manifestation of a deeper symmetry between four-dimensional quantum field theories. From this perspective, the categories of boundary conditions associated to the A and B-twists originate from different representations of a category encoding the symplectic geometry of Hamiltonian group actions, the {\it Moore--Tachikawa category} \cite{Moore--Tachikawa}. The traditional geometric formulations of Langlands duality are primarily concerned with the shallowest stratum (the objects) of the above mentioned category. Regarding the deeper strata, or higher morphisms of the Moore--Tachikawa category, a major step was taken by Ben-Zvi--Sakellaridis--Venkatesh \cite{BZSV} through their \textit{relative Langlands program}. Their work initiates the study of the higher categorical structures underlying Langlands duality by singling out a well behaved class of Hamiltonian spaces, termed \emph{hyperspherical}, for which the Langlands duality is described. Hyperspherical Hamiltonian actions are completely determined by a certain group-theoretic piece of data called {\it BZSV triple}, and, out of the latter, Ben-Zvi--Sakellaridis--Venkatesh construct \textit{$L$-sheaves} on the stack of local systems. With it, the authors achieve in {\it op.cit.} a geometrization of $L$-functions, objects of great importance and long history in classical number theory. 

The aim of this paper (and its sequels) is to obtain a $B$-twist formulation of the above mentioned boundary conditions, incorporating the constructions provided by \cite{BZSV}. 

Kapustin--Witten \cite{kapustin&witten} introduced {\it (BBB)-branes} as the appropriate candidates describing boundary conditions within the B-twist. The hyperK\"ahler structure of the Hitchin moduli space is central in the characterization of (BBB)-branes which are, in their simplest incarnation, hyperholomorphic bundles over the Hitchin moduli space. Making use of the Kaledin--Verbitsky twistor correspondence \cite{kaledin&verbitsky} one can reformulate the construction of (BBB)-branes by considering holomorphic bundles on the twistor space satisfying a certain horizontality condition. This is the approach undertaken by the second named author and Hanson \cite{franco&hanson, hanson_1, hanson_2} consisting of two steps. Firstly, the construction of the {\it analytic Deligne moduli stack} by gluing (the analytification of) two copies of the Hodge moduli stack of $\Lambda$-connections, following Deligne--Simpson construction of the twistor space of the Hitchin moduli space \cite{simpson_hodge_2}. Secondly giving rise to a dg-category of (BBB)-branes over the Deligne moduli stack that include and generalize the initial construction of Kapustin--Witten. The (BBB)-branes constructed in {\it op.cit.} arise from the subcategory of analytic GAGA sheaves on the Deligne moduli stack after equipping the complexes with certain decorations related to the so-called {\it horizontal twistor lines}. We recall that, in the work of Hitchin--Karlhede--Lindstr\"om--Ro\v{c}ek \cite{HKLR}, horizontal twistor lines are part of the data necessary for the reconstruction of the hyperK\"ahler structure on the slices of the twistor space.

Building on the ideas of \cite{HKLR}, Katzarkov--Pandit--Spaide and, recently, Kryczka--Tannaka--Yau, advanced in \cite{KPS, kryczka&tannaka&yau} towards a notion of hyperK\"ahler structure in the setting of derived geometry. This task is still incomplete, lacking from an adequate notion of horizontal twistor lines. Never-the-less, Kryczka--Tannaka--Yau introduce the preliminary notion of shifted derived pre-twistor family of hyperKähler type \cite{kryczka&tannaka&yau}, a property that can be applied to the (analytification of the) Deligne moduli stack.

\subsubsection*{Summary of the paper}

Our first contribution is the introduction of the notion of {\it quasi-algebraic} (derived) stacks and their associated categories of sheaves. This notion interpolates between algebraic and analytic geometry, and the associated sheaf categories inherit essential properties from their algebraic underpinning. For instance, certain finiteness properties of Hom-spaces, the adjoint pair of shriek-pullback and pushforward functors, as well as base change theorems and the projection formula. This is crucial, as the constructions described below will make use of these tools and their relations.   

We then apply the quasi-algebraic framework to the various moduli stacks of non-abelian Hodge theory, obtaining the {\it Deligne moduli stack} and its relative analogue. The latter is a particular class of quasi-algebraic derived stacks that play the role of the total space of hyperholomorphic bundles over twistor space. The relative Deligne moduli stacks are naturally equipped with a morphism $\theta$ to the (absolute) Deligne moduli stack, pushing along which we obtain the sought-after (BBB)-branes. 

In the central part of the paper we address the construction of a version of the Moore--Tachikawa category, whose members are, roughly speaking, equipped with pre-quantization data. We define the {\it polarized Moore--Tachikawa} $2$-category $\Cc$ by considering the same objects as the Moore--Tachikawa category ({\it i.e.} reductive groups) and we stipulate that morphisms come with \textit{polarization data} analogous to those arising in the context of geometric quantization. Our main result provides a B-twist representation of $\Cc$ in dg-categories: 

\begin{theorem_A}[Theorem \ref{th representation of Cc}]
Let $C$ be a smooth projective curve. There exists a representation $\bB$ of the polarized Moore--Tachikawa $2$-category $\Cc$ sending their objects $G \in \Ob(\Cc)$ to the dg-category of quasi-algebraic sheaves on the Deligne moduli stack with structure group $G$ over the curve $C$. 
\end{theorem_A}

We highlight that the use of quasi-algebraic sheaves is key for {\bf Theorem A}. This is so as the analytic framework (like the one used in previous work of the second author \cite{franco&hanson}) lacks from the tools necessary for such task.

The image under $\bB$ of a $1$-morphism is obtained by integral functors whose kernels are provided by push-forward under the above mentioned morphisms $\theta$. In our second main result, these kernels are shown to restrict to the $L$-sheaves appearing in \cite{BZSV} whenever they arise from a BZSV triple. On the other hand, when the BZSV data is simply a representation of cotangent type, these kernels restrict to Gaiotto's (BBB)-branes.

\begin{theorem_B}[Corollaries \ref{co relation with BZSV} and \ref{co relation with Gaiotto}]
Let $M$ be a $1$-morphism in $\Cc$ associated to a BZSV triple with structure group $G$. 
\begin{enumerate}
    \item The restriction of $\bB(M)$ to the stack of $G$-local systems $\Loc_G(C)$ is coincides with the corresponding (unnormalized) BZSV $L$-sheaf.
    \item Suppose $M = T^*V$ is a $G$-representation of cotangent type. The restriction of $\bB(M)$ to the stack of $G$-Higgs bundles $\Higgs_G(C)$ coincides with Gaiotto's (BBB)-branes (denoted $\mathcal{V}(C,G,M)$ in \cite{Gaiotto}). 
\end{enumerate}
\end{theorem_B}

These constructions fit naturally into our broader program of constructing the B-twist of the Langlands TQFT, where we shall enhance the target of $\bB$ to a category of quasi-algebraic sheaves with hyperholomorphic structure.

\begin{conjecture_A}
There exists a representation of the polarized Moore--Tachikawa $2$-category\footnote{More precisely, one should restrict to a certain quasi-affine subcategory. See Section \ref{subsect next steps} for more discussion on this point.} $\Cc$ sending their objects $G \in \Ob(\Cc)$ to the dg-category of (BBB)-branes over the Deligne moduli stack associated to $G$. 
\end{conjecture_A}

We conclude by discussing how the dg-category of (BBB)-branes can be built out of the category of quasi-algebraic sheaves, once one distinguishes an appropriate class of horizontal twistor lines (hence completing the twistor structure of the Deligne moduli stack). We will address {\bf Conjecture A} in subsequent work, building off of {\bf Theorem A}.

\subsubsection*{Organization of the paper}

The paper is organized as follows. Quasi-algebraic stacks and their sheaves are introduced in Section \ref{sc qalg}. In Section \ref{section nonabelian Hodge review} we provide the construction of the relative Deligne moduli stack. The category $\Cc$ is defined in Section \ref{sc polarized MT}. Our main result amounts to Theorem \ref{th representation of Cc}, whose proof spans Sections \ref{sc 1-mor} and \ref{sc 2-mor}. Section \ref{sc BZSV} contains our second main result, Corollaries \ref{co relation with BZSV} and \ref{co relation with Gaiotto}, providing the relation of our representation with the $L$-sheaves appearing in \cite{BZSV} and the (BBB)-branes of \cite{Gaiotto}. In Section \ref{sc BBB-branes} we discuss the construction of (BBB)-branes and explain to some extent a sequel to the present work treating {\bf Conjecture A}.

\subsubsection*{Acknowledgements}
The authors would like to thank David Ben-Zvi, Robert Hanson, Tamás Hausel, Hiraku Nakajima, and Dimitri Wyss for inspiring and instructive conversations that refined our understanding of (BBB)-branes, leading to this work. 

\section{Quasi-algebraic sheaves}
\label{sc qalg}

\subsection{Quasi-algebraic stacks} \label{subsect: gluingOnDeligneStack}

Following Deligne, Simpson constructed \cite{simpson_hodge_2} the twistor space of the moduli space of Higgs bundles by gluing the Hodge moduli space over a curve $\Cc$ and that of its complex conjugate curve $\ol{\Cc}$. A similar gluing construction between the Hodge stacks is used in \cite{franco&hanson} to define the Deligne moduli stack. As we eventually seek to generalize this construction, we introduce in this section the notion of {\it quasi-algebraic} stacks which formalizes the analytic gluing of algebraic stacks. We then define quasi-algebraic sheaves on quasi algebraic stacks and study their properties and their behavior under the six-functor formalism.

We denote by $\mathrm{dAff}$ the $\infty$-category of derived affine scheme almost of finite presentation and by $\mathrm{dSt}$ the $\infty$-category of derived stacks. The $\infty$-categories $\mathrm{dAn}$ and $\mathrm{dAnSt}$ of derived analytic spaces and derived analytic stacks, respectively, were developped in \cite{porta_GAGA, lurie_11_c, porta&yue_dAn}. Many of the structural properties of $\mathrm{dAff}$ and $\mathrm{dSt}$ admit natural analogues in their analytic counterparts $\mathrm{dAn}$ and $\mathrm{dAnSt}$. The work of Holstein and Porta \cite[Section 3]{holstein&porta} provides a derived analog of the usual analytification functor, giving rise to the functors
\[
(\cdot)^{\an} : \mathrm{dAff} \to \mathrm{dAn} \quad \text{and} \quad (\cdot)^{\an} : \mathrm{dSt} \to \mathrm{dAnSt},
\]
which preserve colimits and finite limits.

\begin{definition}
A {\it quasi-algebraic stack} is a tuple $Z = (Z_0, \Rr, \mu_1, \mu_2)$, where $\Rr$ is an analytic stack, $Z_0$ is an algebraic stack, and 
\[\begin{tikzcd}
	\Rr & {Z_0^{\mathrm{an}}}
	\arrow["{\mu_2}"', shift right, from=1-1, to=1-2]
	\arrow["{\mu_1}", shift left, from=1-1, to=1-2]
\end{tikzcd}
\]
is an (analytically open) equivalence relation of analytic stacks. Naturally associated to a quasi-algebaric stack is its analytification 
$$Z^{\mathrm{an}} := [Z_0^{\mathrm{an}}/\mathcal{R}]$$
defined as the analytic quotient stack of $Z_0^{\mathrm{an}}$ by the equivalence relation $\mathcal{R}$, which is an analytic stack.
\end{definition}

In other words, we may think of a quasi-algebraic stack $Z$ as an analytic stack equipped with a preferred algebraic atlas $Z_0$, with possibly analytic gluing data.

The notion of morphisms of quasi-algebraic stacks is the natural one. 
\begin{definition}
A {\it morphism of quasi-algebraic stacks} $\theta : Z = (Z'_0, \mathcal{R}') \to Z = (Z_0, \mathcal{R})$ is given by a tuple of morphisms $\theta = (\vartheta , \theta_0)$, where $\vartheta : \Rr' \to \Rr$ is a morphism of analytic stacks, $\theta_0: Z'_0 \to Z_0$ is a morphism of algebraic stacks, such that 
\[\begin{tikzcd}
	\Rr' & {(Z'_0)^{\mathrm{an}}} \\
	\Rr & {Z_0^{\mathrm{an}}}
	\arrow["{\mu_2'}"', shift right, from=1-1, to=1-2]
	\arrow["{\mu_1}'", shift left, from=1-1, to=1-2]
	\arrow["\vartheta"', from=1-1, to=2-1]
	\arrow["{\theta_0^{\mathrm{an}}}", from=1-2, to=2-2]
	\arrow["{\mu_2}"', shift right, from=2-1, to=2-2]
	\arrow["{\mu_1}", shift left, from=2-1, to=2-2]
\end{tikzcd}
\]
is a morphism of groupoids in analytic stacks. 
\end{definition}

Our notion of quasi-algebraic stacks $\QAlg$ may be regarded as an instance of a \textit{geometric context} in the terminology of \cite[Definition 2.2]{porta&yu_rep}. 

\begin{definition}
We consider the \textit{quasi-algebraic} site $\QAlg$ of algebraic stacks with the analytically open topology, i.e., objects of $\QAlg$ are algebraic stacks, and a covering in $\QAlg$ of some $X$ is an open covering of analytic stacks $\mathcal{U} \to X^{\mathrm{an}}$. 
\end{definition}

\begin{remark} \label{remark qalg geometric context}
Since representable presheaves on $\QAlg$ are sheaves, we may consider higher geometric stacks in the sense of Definition 2.8 of \textit{op. cit} following \cite{simpson_geometricity}, and quasi-algebraic stacks are of geometricity $\leq 0$.
\end{remark}

By the universal property of quotients, we see that a morphism of quasi-algebraic stacks $\theta: Z \to Z'$ induces a morphism of their analytifications
$$\theta^{\mathrm{an}}: Z^{\mathrm{an}} \to (Z')^{\mathrm{an}}.$$
This functor is faithful, although it is not full, as one can consider strictly analytic morphisms (not induced from the GAGA principle). We say that $\theta$ is an \textit{analytic equivalence} if $\theta^{\mathrm{an}}$ is an equivalence of analytic stacks (in other words, we do not insist that the morphism $\theta_0$ on algebraic atlases is an equivalence).

The fibre product construction extends naturally to the quasi-algebraic framework. 
\begin{definition}
Given two quasi-algebraic morphisms $\theta': Z' \to Z$ and $\theta'' : Z'' \to Z$ we define their {\it quasi-algebraic fibre product} as 
\[
Z' \times_{Z} Z'' := \left ( Z_0' \times_{Z_0} Z_0'', \Rr' \times_{\Rr} \Rr'', \wt \mu_1 , \wt \mu_2 \right ), 
\]
where $\wt \mu_i$ is the morphism provided by the universal property of pull-backs in the commuting diagram below,
\[
\begin{tikzcd}
\Rr' \times_{\Rr} \Rr'' & & \Rr' &
\\
 & Z_0' \times_{Z_0} Z_0'' & & Z_0'
\\
\Rr'' & & \Rr &
\\
 & Z_0'' & & Z_0 .
\arrow[from=1-1, to=1-3]
\arrow[from=1-1, to=3-1]
\arrow[from=1-3, to=3-3]
\arrow[from=3-1, to=3-3]
\arrow[from=2-2, to=2-4]
\arrow[from=2-2, to=4-2]
\arrow[from=2-4, to=4-4, "\theta_0'"]
\arrow[from=4-2, to=4-4, "\theta_0''"']
\arrow[from=1-1, to=2-2, dashed, "\wt \mu_i"]
\arrow[from=1-3, to=2-4, "\mu_i'"]
\arrow[from=3-1, to=4-2, "\mu_i''"']
\arrow[from=3-3, to=4-4, "\mu_i"]
\end{tikzcd}
\]
\end{definition}

Finally, since classical truncation commutes with analytification, it is natural to make the following
\begin{definition}
The {\it truncation} of a derived quasi-algebraic stack $Z = (Z_0, \Rr)$ is the quasi-algebraic stack 
\[
t_0(Z) := \left ( t_0(Z_0), t_0(\Rr),  t_0(\mu_1), t_0(\mu_2) \right ). 
\]
\end{definition}

\begin{remark} \label{rm Alg to QAlg}
Any algebraic stack $Y$ is, in a trivial manner, a quasi-algebraic stack, $Y^{\qalg}=(Y, Y^{\an}, \id, \id)$. Similarly, morphisms between algebraic stacks naturally provide morphisms of quasi-algebraic ones giving rise to the functor $(\cdot)^{\qalg} : \mathrm{dSt} \to \QAlg$. In view of Remark \ref{remark qalg geometric context} we are equivalently considering a morphism of geometric contexts from algebraic stacks with the Zariski topology to $\qalg$.

In general, any particular algebraic atlas $Y_0 \to Y$ gives rise to the quasi-algebraic stack $(Y_0, R^{\an}, p_1^{\an}, p_2^{\an} )$, where $R = Y_0 \times_Y Y_0$ and $p_i$ is the projection onto the $i$-th factor. From the composition $\wt Y_0 \stackrel{f}{\to} Y_0 \to Y$ one can infer a morphism of groupoids $(\wt R \rightrightarrows \wt Y_0) \to (R \rightrightarrows Y_0)$ which induces a morphism of quasi-algebraic stacks 
\[
\nu_f : (\wt Y_0, \wt R^{\an}) \to (Y_0, R^{\an}),
\]
giving rise to an analytic isomorphism. Nevertheless, the above morphism might not have a quasi-algebraic inverse (see Example \ref{ex PP^1 qalg} below, for instance).
\end{remark}

We now have all the ingredients to consider group actions within the quasi-algebraic framework.

\begin{definition}
Given an algebraic group $G$ and a quasi-algebraic stack $Z$, we say that $G^{\qalg} \times Z \to Z$ is an {\it action of $G$ on $Z$} if it satisfies the standard associative and identity axioms up to coherent higher homotopies.
\end{definition}

We shall drop $(\cdot)^{\qalg}$ from the notation of group actions whenever its clear.

\begin{example}
\label{ex PP^1 qalg}
We consider $\P^1$ to be the quasi-algebraic scheme $(\AA^1 \sqcup \AA^1, (\AA^1 - 0)^{\an})$, with the usual gluing datum twisted by a minus sign
\[\begin{tikzcd}
	{(\mathbb{A}^1-0)^{\mathrm{an}}} & {\mathbb{A}^1 \sqcup \mathbb{A}^1}
	\arrow["{\lambda \mapsto -\lambda^{-1}}"', shift right, from=1-1, to=1-2]
	\arrow["{\lambda \mapsto \lambda}", shift left, from=1-1, to=1-2].
\end{tikzcd}\]
We will denote by $\AA^1_{\Hod}$ and $\AA^1_{\ol \Hod}$, respectively, the first and second factors of the atlas of the quasi-algebraic scheme $\PP^1$, both isomorphic to $\AA^1$.

There is an evident quasi-algebraic morphism
\begin{equation} \label{eq nu morphism}
\nu : \P^1 = (\AA^1 \sqcup \AA^1, (\AA^1-0)^{\mathrm{an}}) \to (\PP^1)^{\qalg} = (\PP^1, (\PP^1)^{\mathrm{an}})
\end{equation}
which is an analytic equivalence, $(\P^1)^{\an} = (\PP^1)^{\an}$, with no quasi-algebraic inverse, as any morphism $\PP^1 \to \AA^1$ is forcely constant.

The usual conjugation $\lambda \mapsto \ol\lambda$ in $\AA^1 - 0$ extends to the real structure 
\begin{equation} \label{eq chi P^1}
\chi_{\P^1} : \P^1 \to \P^1 
\end{equation}
provided by the antipodal map, $\lambda \mapsto - \ol{\lambda}^{-1}$. Note that $\chi$ has no fixed points, differing from the usual conjugation in $\PP^1$, which fixes the unit circle. The pair $((\P^1)^{\an}, \chi_{\P^1}^{\an})$ appears in the literature under the name of {\it twistor line}.
\end{example}

\subsection{Categories of quasi-algebraic sheaves} 

Following \cite{BZSV}, given a derived stack $X$ we denote by $\QC(X)$ the dg-enhanced derived category of quasi-coherent sheaves. Consider the dg-subcategory $\Coh(X) \subset \QC(X)$ of bounded complexes with coherent cohomology and let us denote by $\QC^!(X)$ the ind-completion of $\Coh(X)$. The category $\QC^!(X)$ is monoidal, with monoidal structure $\otimes^!:= \Delta^!( \, \cdot \, \boxtimes \, \cdot \, )$, where $\Delta: X \to X \times X$ is the diagonal morphism.

Since quasi-algebraic stacks admit an algebraic cover, it is natural to consider certain analytic sheaves equipped with algebraic structure when pulled back to its atlas.

\begin{definition}
 Consider a quasi-algebraic stack $Z = (Z_0, \Rr, \mu_1, \mu_2)$. We define the dg-category $\QA^!(Z)$ of {\it quasi-algebraic ind-coherent sheaves on} $Z$ as the homotopy limit in dg-categories along the diagram
\begin{equation}
\label{eq def qalg indcoh}
\begin{tikzcd}
	{} & {\QC^!(\Rr)} & {\QC^!(Z_0)}
	\arrow["{\mu_2^! \circ ( \, \cdot \, )^{\mathrm{an}}}", shift left, from=1-3, to=1-2]
	\arrow["{\mu_1^!\circ ( \, \cdot \, )^{\mathrm{an}}}"', shift right, from=1-3, to=1-2]
\end{tikzcd}.
\end{equation}
\end{definition}

\begin{remark}
    In principle, one can imagine defining a category of quasi-algebraic quasi-coherent sheaves using the same diagram, with $!$-pullbacks replaced by $*$-pullbacks. However, it appears that the quasi-coherent theory of analytification presents serious difficulties which we will not address in the current discussion. 
\end{remark}

Note that the canonical sheaf $\omega_Z = (\omega_{Z_0}, \gamma)$ is a well-defined object of $\QA^!(Z)$, where $\gamma$ is the canonical isomorphism $\mu_1^! \omega_{Z_0}^{\mathrm{an}} \simeq \mu_2^!\omega_{Z_0}^{\mathrm{an}}$ is naturally an object of $\QA^!(Z)$.

Since the diagram \eqref{eq def qalg indcoh} factors through the analytification functor $( \, \cdot \, )^{\mathrm{an}}: \QA^!(Z_0) \to \QA^!(Z_0^{\mathrm{an}})$, by descent of analytic sheaves and the universal property of $\QA^!(Z)$ there is a canonical analytification functor on quasi-algebraic sheaves
\[
(\,  \cdot \, )^{\mathrm{an}}: \QA^!(Z) \to \QA^!(Z^{\mathrm{an}})
\]
which simply sends a quasi-algebraic sheaf $\Ee = (\Ee_0, \gamma)$ to $(\Ee_0^{\mathrm{an}}, \gamma)$ viewed as descent datum for an analytic sheaf on $Z^{\mathrm{an}}$.

Recall the functor $(\cdot)^{\qalg} : \mathrm{dSt} \to \qalg$ sending the stack $Y$ to the quasi-algebraic stack $Y^{\qalg} = (Y, Y^{\an}, \id^{\an}, \id)$, and note that there exists a corresponding functor between the associated categories of sheaves,
\[
(\cdot)^{\qalg} : \QC^!(Y) \to \QA^!(Y^{\qalg}),
\]
sending $\Ee$ to the pair $(\Ee, \id)$. By the properties of the analytification functor, $(\cdot)^{\qalg}$ is faithful, but fails to be full or essentially surjective.

\begin{remark} 
    We observe that, by the GAGA principle, the analytification functor is conservative if $Z^{\mathrm{an}}$ is a proper analytic space. In that case, $(\cdot)^{\qalg}$ inherits this property too.
\end{remark}

Given $\Ee = (\Ee_0, \gamma), \Ff = (\Ff_0, \gamma')$ in $\QA^!(Z)$ one can define internal Hom-sheaves
\[
\Hhom_Z(\Ee,\Ff) := \left(\Hhom_{Z_0}(\Ee_0, \Ff_0), \gamma'' \right ), 
\]
as objects of $\QA^!(Z)$, where $\gamma''$ is the natural identification 
$$\mu_1^!\Hhom_{Z_0}(\Ee_0, \Ff_0)^{\mathrm{an}} \simeq \Hhom_{\Rr}(\mu_1^!\Ee_0^{\mathrm{an}}, \mu_1^!\Ff_0^{\mathrm{an}}) \simeq \Hhom_{\Rr}(\mu_2^!\Ee_0^{\mathrm{an}}, \mu_2^!\Ff_0^{\mathrm{an}}) \simeq \mu_2^!\Hhom_{Z_0}(\Ee_0, \Ff_0)^{\mathrm{an}},$$
where, in the first and third equivalences we appeal to the projection formula (see statement (6) of \cite[pg.276]{gaitsgory&rozenblyum_1}) and the fact that $\mu_i^! = \mu_i^*$ for both $i = \{1,2\}$ as they are open immersions.

Similarly, we set
\[
\Ee \otimes^! \Ff := (\Ee_0 \otimes^! \Ff_0, \gamma'''),
\]
where $\gamma'''$ is the natural identification
$$\mu_1^!(\Ee_0^{\mathrm{an}} \otimes^! \Ff_0^{\mathrm{an}}) \simeq \mu_1^!\Ee_0^{\mathrm{an}} \otimes^! \mu_1^!\Ff_0^{\mathrm{an}} \simeq \mu_2^!\Ee_0^{\mathrm{an}} \otimes^! \mu_2^!\Ff_0^{\mathrm{an}} \simeq \mu_2^!(\Ee_0^{\mathrm{an}} \otimes^! \Ff_0^{\mathrm{an}}).$$
In the first and third equivalences we appeal to the $\otimes^!$-monoidality of the $!$-pullback, thus defining a monoidal structure, 
\[
\otimes^! : \QA^!(Z) \times \QA^!(Z) \to \QA^!(Z).
\]
compatible with the usual monoidal structure on $\QA^!(Z_0)$ and $\QA^!(Z^{\mathrm{an}})$. With respect to these monoidal structures, $\omega_Z$ is the monoidal unit for $\otimes^!$. In order to simplify notation, we sometimes write $\mathbf{1}_Z \in \QA^!(Z)$ to refer to the $!$-monoidal unit.

One can equip the categories of quasi-algebraic sheaves with push-forwards and pull-backs under quasi-algebraic morphisms as follows.

\begin{lemma} \label{lm qalg push and pull}
Associated to the morphism of quasi-algebraic stacks $\theta : Z \to Z'$, there exist functors
\[
\theta_{*} : \QA^!(Z) \to \QA^!(Z')
\]
and
\[
\theta^! : \QA^!(Z') \to \QA^!(Z),
\]
compatible with $\vartheta_*$ and $\theta_{0, *}$, and with $\vartheta^!$ and $\theta_{0}^!$, respectively. Furthermore, one has the adjuntion
\begin{equation} \label{eq qalg adjunction}
\theta_* \dashv \theta^!.
\end{equation}
\end{lemma}

\begin{proof}
The morphism of diagrams
\[\begin{tikzcd}
	{\QC^!(Z_0)} & {\QC^!(Z_0')} \\
	{\QC^!(\Rr)} & {\QC^!(\Rr')}
	\arrow["{\theta_{0,*}}", from=1-1, to=1-2]
	\arrow[shift right, from=1-1, to=2-1]
	\arrow[shift left, from=1-1, to=2-1]
	\arrow[shift right, from=1-2, to=2-2]
	\arrow[shift left, from=1-2, to=2-2]
	\arrow["{\vartheta_*}", from=2-1, to=2-2]
\end{tikzcd}\]
induces, by universal property of $\QA^!(Z')$, a functor $\theta_*: \QA^!(Z) \to \QA^!(Z')$.

The construction of $\theta^!$ is analogous. Replacing $\theta_{0,*}$ and $\vartheta_*$ in the previous diagram by $\theta_0^!$ and $\vartheta^!$, respectively, and appealing now to the universal property of $\QA^!(Z)$, we obtain a functor $\theta^!: \QA^!(Z') \to \QA^!(Z)$.

Finally, adjunction in the quasi-algebraic setting follows from the corresponding algebraic statement and the fact that the analytification functor is conservative.
\end{proof}

\begin{example} \label{ex line bundles on PP^1 twist}
By the above, one can construct for every $n \in \ZZ$ the line bundle $\Oo_{\P^1}(n) \in \QA^!(\P^1)$ by pulling-back the corresponding line bundle under \eqref{eq nu morphism},
\[
\Oo_{\P^1}(n) := \nu^!\Oo_{\PP^1}(n)^{\qalg}.
\]
\end{example}

The relations between the four {\it quasi-algebraic functors} $\theta^!$, $\theta_*$, $\Hom_Z$ and $\otimes^!$ are naturally inherited by the analog relations of their algebraic counterparts. 

\begin{lemma}[Quasi-algebraic base change]
For a Cartesian square of quasi-algebraic stacks
\[
\begin{tikzcd}
Z' & Z
\\
Y' & Y
\arrow[from=1-1, to=1-2, "\rho'"]
\arrow[from=1-1, to=2-1, "\theta'"']
\arrow[from=1-2, to=2-2, "\theta"]
\arrow[from=2-1, to=2-2, "\rho"'],
\end{tikzcd}
\]
one has the isomorphism  
\[
\theta_*\rho^! \cong (\rho')^! \theta'_*
\]
of functors from $\QA^!(Z)$ to $\QA^!(Y')$.
\end{lemma}

\begin{proof}
This follows naturally from the functorial identification $\theta \circ \rho' \cong \rho \circ \theta'$ and the quasi-algebraic adjunction \eqref{eq qalg adjunction}.
\end{proof}

\begin{lemma}[Quasi-algebraic projection formula]
For any morphism of quasi-algebraic stacks $\theta : Z \to Y$, given $\Ee \in \QA^!(Z)$ and $\Ff \in \QA^!(Y)$, one has the isomorphism
\[
\theta_* \left ( \Ee \otimes^! \theta^! \Ff \right ) \cong \left ( \theta_* \Ee \right ) \otimes^! \Ff.
\]
\end{lemma}

\begin{proof}
The proof follows immediately from corresponding statement in the algebraic setting together with the fact that the analytification functor is conservative.
\end{proof}

\subsection{Shearing and the exponential sheaf}

Our goal in this section is to describe a certain quasi-algebraic sheaf which will be of eventual importance to us, called the \textit{exponential sheaf}. After describing Koszul duality in terms of the so-called {\it shear functor}, Ben-Zvi--Sakellaridis--Venkatesh  introduced this sheaf as the Koszul dual of a sky-scraper sheaf over $\AA^1$. 

We first extend the shear functor to the quasi-algebraic setting, recalling, along the way, its original formulation provided by \cite{BZSV} in the algebraic context. Remaining still in this framework, we extend to the total spaces of tangent and cotangent bundles over $\PP^1$ the funtors giving rise to Koszul duality and the construction of the exponential sheaf. Finally, we show how the latter induces its quasi-algebraic analog.

\subsubsection{Quasi-algebraic shearing} We will need to introduce the technical device of \textit{shearing} \cite[Section 6]{BZSV} in the quasi-algebraic context, which we will restrict to the simpler case of even weight actions.

Let $Z = (Z_0, \Rr, \mu_1, \mu_2)$ be a quasi-algebraic stack with even $\GG_m$-action, i.e., that the induced $\GG_m$-action on $Z_0$ factors through the squaring map $x \mapsto x^2$. Suppose $\Ee, \Ff \in \QA^!(Z)$ are two quasi-algebraic sheaves on $Z$ equipped with a quasi-algebraic morphism $f :\Ee \to \Ff$. This includes the data of
\begin{itemize}
    \item ind-coherent sheaves $\Ee_0$ and $\Ff_0 \in \mathrm{QC}^!(Z_0)$ with an algebraic morphism $f_0: \Ee_0 \to \Ff_0$,
    \item and identifications
    \[
    \mu_1^!\Ee^{\an}_0 \overset{\sim}{\to} \mu_2^! \Ee^{\an}_0  \quad \text{and} \quad \mu_1^!\Ff^{\an}_0 \overset{\sim}{\to} \mu_2^! \Ff^{\an}_0
    \]
    such that the diagram
    \begin{equation}
        \begin{tikzcd}
	{\mu_1^!\Ee^{\an}_0} & {\mu_2^!\Ee^{\an}_0} \\
	{\mu_1^!\Ff^{\an}_0} & {\mu_1^!\Ff^{\an}_0}
	\arrow[from=1-1, to=1-2]
	\arrow["{\mu_1^!f^{\an}_0}"', from=1-1, to=2-1]
	\arrow["{\mu_1^!f^{\an}_0}", from=1-2, to=2-2]
	\arrow[from=2-1, to=2-2]
\end{tikzcd}
    \end{equation}
    commutes. 
\end{itemize}

Ben-Zvi--Sakellaridis--Venkatesh introduced the \textit{sheared category} $\QC^!(Z_0)^\shear$ as the dg-category with the same underlying objects as $\QC^!(Z_0)$, but with {\it sheared} homomorphisms \cite[Section 6.1]{BZSV}
\[
    \mathrm{Hom}(\Ee_0,\Ff_0)^\shear := \bigoplus_{k \in \mathbf{Z}} \mathrm{Hom}(\Ee_0,\Ff_0)_k[k].
\]
where $(-)_i$ denotes the $\GG_m$-weight space of weight $k$. This gives rise to the {\it shear functor}
\[
\QC^!(Z_0) \ni \Ee \longmapsto \Ee^\shear \in \QC^!(Z_0)^\shear.
\]

Since the morphisms $f_0$ live in the algebraic category $\QC^!(Z_0)$, the $\GG_m$-action on $\mathrm{Hom}_{Z_0}(\Ee_0,\Ff_0)$ induces a $\GG_m$-action on $\mathrm{Hom}_Z(\Ee,\Ff)$ which decomposes this space into a direct sum over even $\GG_m$-weights\footnote{Note that this would not be the case for the category of analytic ind-coherent sheaves $\QC^!(Z^{\an})$ with the action of $(\GG_m)^{\an}$: the algebraic structure underlying quasi-algebraic sheaves is crucial if $Z$ is not projective.}. This allow us to extend the above to the quasi-algebraic framework.

\begin{definition}
Given a quasi-algebraic stack $Z$ equipped with an even $\GG_m$-action, define the \textit{sheared category} $\QA^!(Z)^\shear$ of quasi-algebraic sheaves on $Z$ to be the dg-category with the same underlying objects as $\QA^!(Z)$, and morphisms
\[
\mathrm{Hom}(\Ee,\Ff)^\shear := \bigoplus_{k \in \mathbf{Z}} \mathrm{Hom}(\Ee,\Ff)_k[k].
\]
We will write 
$$\QA^!(Z) \ni \Ee \longmapsto \Ee^\shear \in \QA^!(Z)^\shear$$
for corresponding objects under shearing. 
\end{definition}

We record some evident functorial properties of the shearing operation. 

\begin{remark}
If $f: X \to Y$ is a $\GG_m$-equivariant morphism of quasi-algebraic stacks, then the pushforward and exceptional pullback functors $f_*, f^!$ induce analogues on the sheared categories $f_*^\shear: \QA^!(X)^\shear \to \QA^!(Y)^\shear$ and $(f^!)^\shear: \QA^!(Y)^\shear \to \QA^!(X)^\shear$, and we may drop the superscript $(-)^\shear$ on the functors when the context is clear. 
\end{remark}

\begin{remark}
When $Z$ is equipped with the trivial $\GG_m$-action, the sheared category of sheaves is canonically equivalent to the usual one, and we denote by
\begin{equation}
    (-)^\unshear: \QA^!(Z)^\shear \to \QA^!(Z)
\end{equation}
this canonical identification. 
\end{remark}

\subsubsection{Koszul duality and line bundles over the projective line}

The introduction of shearing allowed Ben-Zvi--Sakellaridis--Venkatesh to provide a precise formulation of Koszul duality on the affine line \cite[Appendix A]{BZSV}. Our goal in this section is to provide a functor, covering the total spaces of tangent and cotangent bundles over $\PP^1$, which restricts pointwise to the functor giving rise to Koszul duality.

Considering $\AA^1$ as $\Spec(k[x_0])$ with $x_0$ of weight $2$, one has that $k[x_0]^{\shear} = k[x_2]$, with $x_2$ having cohomological degree $-2$. One has the identification of module categories
\[
k[x_0]-\mathrm{mod} = \QC(\AA^1) \quad \text{and} \quad k[x_0]^{\shear}-\mathrm{mod} = \QC(\AA^1)^{\shear}.
\]
Consider the augmentation object to be $k_0$, the sky-scraper at $0 \in \AA^1$, and its sheared version $k_0^{\shear}$. Observe that the {\it Koszul dual} algebra amounts to
\[
k[y_{-1}] = \End_{k[x_0]^{\shear}}\left ( k_0^{\shear}, k_0^{\shear} \right )^{\mathrm{op}},
\]
with $y_{-1}$ in cohomological degree $1$. Note that $\Spec(k[y_{-1}]) = \AA^1[-1]$. One can define the {\it Morita functors}
\[
(\cdot) \otimes_{k[x_0]^{\shear}} k_0^{\shear} : k[x_0]^{\shear}-\mathrm{mod} \leftrightarrow k[y_{-1}]-\mathrm{mod} : \Hom_{k[y_{-1}]}\left (k^{\shear}, (\cdot) \right ),
\]
sending perfect complexes of $k[x_0]^{\shear}$-modules to bounded complexes of $k[y_{-1}]$-modules with coherent cohomology
\[
\Perf(\AA^1)^\shear \simeq \Coh(\mathbb{A}^1[-1]).
\]
Passing to their $\mathrm{Ind}$-completions one obtains
\[
\QC(\AA^1)^\shear \simeq \QC^!(\mathbb{A}^1[-1]).
\]
as described in \cite[pg. 363]{BZSV}. Furthermore, since the $\GG_m$-weights of $x$ and $y$ are opposite, shearing by the inverse of the previous action, the above gives rise to \cite[pg. 365]{BZSV}
\begin{equation} \label{eq Koszul duality over fibres}
\Kk_{\AA^1} : \QC(\AA^1) \stackrel{\simeq}{\to} \QC^!(\mathbb{A}^1[-1])^\shear.
\end{equation}

Let $\O(2)$ and $\O(-2)$ denote, respectively, the total space of the tangent $T_{\PP^1} \cong \Oo_{\PP^1}(2)$ and cotangent $T^*_{\PP^1} \cong \Oo_{\PP^1}(-2)$ bundles of $\PP^1$. Let us equip $\O(-2) \to \PP^1$ with a $\GG_m$-action of weight $2$ extending to the trivial action on $\PP^1$. Note that 
\[
\O(2) = \Spec_{\PP^1} \left ( \Sym^\bullet \Oo_{\PP^1}(-2) \right ), \quad \text{and} \quad \O(-2) = \Spec_{\PP^1} \left ( \Sym^\bullet \Oo_{\PP^1}(2) \right ),
\]
so, $\GG_m$ acts on the fibres of $\Oo_{\PP^1}(-2)$ and $\Oo_{\PP^1}(2)$ with weights $-2$ and $2$, respectively. With respect to this action, set
\[
S := \Sym^\bullet \Oo_{\PP^1}(2) \quad \text{and} \quad S^{\shear} := \Sym^\bullet \Oo_{\PP^1}(2)^{\shear},
\]
and consider the module categories
\[
S-\mathrm{mod} = \QC(\O(-2)) \quad \text{and} \quad S^{\shear}-\mathrm{mod} = \QC(\O(-2))^{\shear}.
\]
Let us consider in $\QC(\O(-2))^{\shear}$ the augmentation object $\delta_0^{\shear}$, where $\delta_0$ amounts to the push-forward of $\Oo_{\PP^1}$ under the $0$-section $\PP^1 \to \O(-2)$. Recall that $\delta_0$ is resolved by $q^*\Oo_{\PP^1}(2) \stackrel{s}{\to} \Oo_{\O(-2)}$, where $s$ is the tautological section of $q^*\Oo_{\PP^1}(-2)$ over $\O(-2)$. With such a choice, set
\[
\Lambda := \End_{S^{\shear}} \left (\delta_0^{\shear}, \delta_0^{\shear} \right )^{\mathrm{op}}, 
\]
and observe that 
\[
\Lambda \cong \Sym^\bullet \left ( \Oo_{\PP^1}(-2)[-1] \right ), 
\]
so, its relative $\Spec$ gives rise to
\[
\Spec_{\PP^1} \Lambda \cong \O(2)[-1],
\]
the self-intersection of the $0$-section within the total space of the cotangent bundle of $\PP^1$. Consider as well the Morita functors
\[
(\cdot) \otimes_{S^{\shear}} \delta_0^{\shear} : S^{\shear}-\mathrm{mod} \leftrightarrow \Lambda-\mathrm{mod} : \Hom_{\Lambda}\left (\delta_0^{\shear}, (\cdot) \right ).
\]
Since $\delta_0^{\shear}$ is perfect, its tensor product preserves coherence,
\[
(\cdot) \otimes_{S^{\shear}} \delta_0^{\shear} : \Coh(\O(-2))^{\shear} \cong \Perf(\O(-2))^{\shear} \to \Coh(\O(2)[-1]),
\]
where we recall that $\O(2)$ is smooth, promoting to their corresponding $\Ind$-completions,
\[
(\cdot) \otimes_{S^{\shear}} \delta_0^{\shear} : \QC(\O(-2))^{\shear} \to \QC^!(\O(2)[-1]),
\]
Finally, negating the $\GG_m$-action, provides a functor
\begin{equation} \label{eq Koszul duality over O2}
\Kk_{\O(-2)} : \QC(\O(-2)) \to \QC^!(\O(2)[-1])^{\shear}.
\end{equation}

We next show that \eqref{eq Koszul duality over O2} restricts pointwise to \eqref{eq Koszul duality over fibres}.

\begin{lemma} \label{lm pointwise restriction of Koszul}
Consider the closed immersions $\imath_x:\AA^1 \hookrightarrow \O(-2)$ and $i_x:\AA^1[-1] \hookrightarrow \O(2)[-1]$ of the fibres of the projections $q : \O(-2) \to \PP^1$ and $\dot{q} : \O(2)[-1] \to \PP^1$ over the same point $x \in \PP^1$ ({\it i.e. } the image of the compositions are $\im(q \circ \imath_x) = \{ x \} = \im(\dot{q} \circ i_x)$). Then, 
\[
\Kk_{\AA^1}\left (\imath_x^*(\cdot) \right ) \cong i_x^*\Kk_{\O(-2)}(\cdot).
\]
\end{lemma}

\begin{proof}
Note that $\imath_x^* \delta_0 \cong k_0$, where we recall that the latter denotes the sky-scraper at $0 \in \AA^1$. Then, for every coherent $\Ff$ one has
\[
k_0 \otimes_{k[x_0]} \imath_x^* \Ff \cong \imath_x^* \delta_0 \otimes_{k[x_0]} \imath_x^*\Ff \cong i_x^* \left (\delta_0 \otimes_{S} \Ff \right ).
\]
The proof follows immediately from the sheared version of the above identification.
\end{proof}

\subsubsection{The quasi-algebraic exponential sheaf}
\label{sc qa exponential sheaf}

In this section we complete the construction of the exponential sheaf over the $[-1]$-shifted quasi-algebraic total space of the cotangent of $\PP^1$.  

Ben-Zvi--Sakellaridis--Venkatesh \cite[Appendix A.2]{BZSV} defined the {\it exponential sheaf} in $\QC^!(\AA^1[-1])^\shear$ as the Koszul dual of $k_1$, the sky-scraper at $1 \in \AA^1$,
\[
\exp_{\AA^1} := \Kk_{\AA^1}(k_1).
\]

Recall the tautological section $s$ of over $\O(-2)$, and consider the associated complex $q^*\Oo_{\PP^1}(2) \stackrel{s}{\to} \Oo_{\O(-2)}$. We define the {\it extended (algebraic) exponential sheaf} to be the image of this complex under the functor \eqref{eq Koszul duality over O2},
\[
\exp_{\O(-2)} := \Kk_{\O(-2)} \left (  q^*\Oo_{\PP^1}(2) \stackrel{1-s}{\to} \Oo_{\O(-2)}  \right ) \in \QC^!(\O(2)[-1]).
\]
Since $k_1$ is quasi-isomorphic to $k[x_0] \stackrel{1-x_0}{\to} k[x_0]$, as an immediate consequence of the above construction and Lemma \ref{lm pointwise restriction of Koszul}, one identifies 
\begin{equation} \label{eq justification of exp 1}
i_x^*\exp_{\O(-2)} \cong \exp_{\AA^1}.
\end{equation}
for the inclusion $i_x^* : \AA^1[-1] \hookrightarrow \O(2)[-1]$ of the fibre over any $x \in \PP^1$.

To up-grade the above to the quasi-isomorphic framework, consider the quasi-algebrification of $\O(2)[-1]$ and the exponential sheaf,
\[
\exp_{\O(-2)}^{\qalg} \in \QA^!(\O(2)[-1]^{\qalg}).
\]
Recalling the morphism $\nu : \P^1 \to (\PP^1)^{\qalg}$ defined in \eqref{eq nu morphism}, we consider the fibre product
\[
\OO(2)[-1] := \O(2)[-1]^{\qalg} \times_{(\PP^1)^{\qalg}} \P^1,
\]
which is naturally equipped with the structural projection
\[
\mathbf{q} : \OO(2)[-1] \to (\PP^1)^{\qalg}.
\]
Finally, with all of this, we define the {\it the quasi-algebraic exponential sheaf}
\[
\mathbf{exp} := \mathbf{q}^! \exp_{\O(-2)}^{\qalg} \in \QA^!(\OO(2)[-1]).
\]

Note that \eqref{eq justification of exp 1} connects the previous construction with the one provided by Ben-Zvi--Sakellaridis--Venkatesh \cite[Appendix A.2]{BZSV}. Furthermore, denoting by $\mathbf{i}_x : \AA^1[-1]^{\qalg} \to \OO(2)[-1]$ the morphism induced by $i_x$ over any point $x \in \PP^1$, functoriality in \eqref{eq justification of exp 1} provides
\[
\mathbf{i}_x^! \mathbf{exp} \cong \exp_{\AA^1}^{\qalg},
\]
which justifies our construction.

\section{Twistor stacks from nonabelian Hodge theory} \label{section nonabelian Hodge review}

Following Deligne, Simpson constructed \cite{simpson_hodge_2} the twistor space of the moduli space of Higgs bundles by gluing the Hodge moduli space over a curve $C$ and that of its complex conjugate curve $\ol{C}$. In \cite{franco&hanson} the Deligne moduli stack is defined by performing a similar gluing construction between the Hodge stacks of a curve and its conjugate, $\Hodge_{G}(C)^{\an}$ and $\Hodge_{G}(\ol{C})^{\an}$, understood as certain mapping stacks. We seek to generalize the work of \cite{franco&hanson} to the quasi-algebraic setting as well as in the relative direction, substituting the classifying spaces by more general quotient stacks in the target of the corresponding mapping stacks.

\subsection{Simpson shapes and their mapping stacks}
\label{sc simpson shapes and mapping stacks}

Let $C$ be a smooth projective curve over $\CC$. Simpson constructed a series of (classical) 1-stacks $C_{\Dol}$, $C_{\dR}$, $C_{\Hod}$ and $C_{\B}$, the \textit{Simpson shapes} associated to $C$, intrinsically encoding non-abelian Hodge theory in their categories of coherent sheaves. We quickly review these constructions here for the reader's convenience.

The \textit{Dolbeault shape} $C_{\Dol}$ is defined \cite{simpson_dolbeault, simpson_dolbeault_2} as the relative classifying stack for the formal completion of the tangent bundle along the zero section, $\widehat{TC} \to C$. Similarly, the \textit{de Rham shape} $C_{\dR}$ is defined by the formal completion of the diagonal embedding $C \hookrightarrow C \times C$. The \textit{Hodge shape} $C_{\Hod}$ is constructed in \cite[§7]{simpson_hodge} by performing a deformation to the normal cone of the morphism $C \to C_{\dR}$. Hence, the Hodge shape comes naturally equipped with a projection $\tau_{\Hod} : C_{\Hod} \to \AA^1$ whose central fibre amounts to the Dolbeault shape, 
\begin{equation} \label{eq Hodge specialises to Dolbeault}
C_{\Hod} \times_{\AA^1} \{0\} \cong C_{\Dol},
\end{equation}
and trivialises to the de Rham shape outside the origin,
\begin{equation} \label{eq Hodge specialises to de Rahm}
\triv : C_{\dR} \times (\AA^1 - 0 ) \xlongrightarrow{\cong} C_{\Hod} \times_{\AA^1} (\AA^1 - 0).
\end{equation}
Finally, the {\it Betti shape} $C_{\B}$ is constructed \cite{simpson_dolbeault} as the constant stack associated to $C^{\, \top}$, where $C^{\, \top}$ denotes the topological space underlying the smooth projective curve $C$. 

The objects above become central in non-abelian Hodge theory when considered as source of {\it mapping stacks}. Recall that we denoted by $\mathrm{dAff}$ the $\infty$-category of derived affine schemes almost of finite presentation and by $\mathrm{dSt}$ the $\infty$-category of derived stacks. Given $S_1, S_2 \in \mathrm{dSt}$, their associated mapping stack $\Maps(S_1,S_2)$ is the derived stack sending the test scheme $T \in \mathrm{dAff}$ to $\Maps_{\mathrm{dSt}}(S_1 \times T, S_2)$, the associated space of maps.

Now let $G$ be any algebraic group. Given a $G$-action $G \acts X$, one can consider the following (derived) mapping stacks over $\CC$
\[
\Higgs_G^X(C) := \Maps(C_{\Dol}, [X/G]),
\] 
\[
\Loc_G^X(C) := \Maps(C_{\dR}, [X/G]),
\]
which are interpolated by the following stack over $\AA^1$
\[
\Hodge_G^X(C) := \Maps_{\AA^1}(C_{\Hod}, [X/G]),
\]
and 
\[
\Rep_G^X(C) := \Maps(C_{\B}, [X/G]).
\]

Note that, when $X = \Spec \CC$, one has the straight-forward identifications of these stacks $\Higgs_G(C)$, $\Loc_G(C)$, $\Hodge_G(C)$ and $\Rep_G(C)$, with the stack of Higgs bundles, integrable connections and integrable $\lambda$-connections over $\CC$ and $G$-representation of the fundamental group, respectively.  

When $G$ is reductive, by the foundational work of Hitchin \cite{hitchin_self} and Simpson \cite{simpson1, simpson2} one has a notion of (semi)stability with respect to which one can construct the associated moduli spaces $\Mm_{\Dol}(C,G)$, $\Mm_{\dR}(C,G)$, $\Mm_{\Hod}(C,G)$ and $\Mm_{\B}(C,G)$ of (topologically trivial) semistable objects. We denote by a superscript $( \, \cdot \, )^{\mathrm{sst}}$ the semistable loci of the moduli stacks $\mathrm{Higgs}_G(C), \mathrm{Loc}_G(C)$, and $\mathrm{Hodge}_G(C)$, which come equipped with surjections to their corresponding semistable moduli spaces
\[
\zeta_{\Dol} : t_0\left (\Higgs_G(C)^{\sst}_0 \right ) \to \Mm_{\Dol}(C,G),
\]
\[
\zeta_{\dR} : t_0\left (\Loc_G(C)^{\sst} \right ) \to \Mm_{\dR}(C,G),
\]
\[
\zeta_{\Hod} : t_0\left (\Hodge_G(C)^{\sst}_0 \right ) \to \Mm_{\Hod}(C,G)_0,
\]
and, recalling that every representation of the fundamental group is semistable,
\[
\zeta_{\B} : t_0\left (\Rep_G(C) \right ) \to \Mm_{\B}(C,G),
\]

giving rise to good moduli spaces in the sense of Alper \cite{alper_good}. Note that we restrict to the topologically trivial components of the Dolbeault and Hodge moduli spaces.

Making use of the mapping stack structure, precomposition by $\triv$ in \eqref{eq Hodge specialises to de Rahm} provides the following description of the restriction of the Hodge stack outside the origin $\{ 0 \} \in \AA^1$,
\begin{equation} \label{eq algebraic Triv map}
\Triv : \Hodge_G^X(C) \times_{\AA^1} (\AA^1 - 0) \xlongrightarrow{\cong} \Maps(C_{\dR} \times (\AA^1 - 0), [X/G]) \cong \Loc_G^X(C) \times (\AA^1 - 0).  
\end{equation}

In reference to Simpson's work, we write $\Sim$ to denote any of the four symbols $\{\Dol, \dR, \Hod, \B \}$. We also abbreviate by $\Sim^X_G$ and $\Sim_G$, the corresponding mapping stacks $\Maps(C_{\Sim}, [X/G])$ and $\Maps(C_{\Sim}, BG)$, respectively.

Recall from Section \ref{subsect: gluingOnDeligneStack} that $\mathrm{dAn}$ and $\mathrm{dAnSt}$ denote the $\infty$-categories of derived analytic spaces and stacks, respectively. Starting from a pair of derived analytic stacks, $\Ss_1, \Ss_2 \in \mathrm{dAnSt}$, one defines their {\it analytic mapping stack}, $\anMaps(\Ss_1, \Ss_2)$, as the derived analytic stack sending the test $\Tt \in \mathrm{dAn}$ to the space $\Maps_{\mathrm{dAnSt}}(\Ss_1 \times \Tt, \Ss_2)$. The work of Holstein--Porta \cite[Theorem 6.14]{holstein&porta} provides the following identification of analytic stacks,
\[
\Maps(S_1,S_2)^{\an} \cong \anMaps(S_1^{\an},S_2^{\an}),
\]
provided that the stack $\Maps(S_1,S_2)$ is geometric, the source $S_1$ satisfies the {\it universal GAGA property} (see \cite[Definition 5.1]{holstein&porta}), and the target $S_2$ is a derived stack locally almost of finite presentation satisfying:
\begin{enumerate}[label=(C\arabic*),ref=C\arabic*]

\item \label{cond C1} $S_2$ is geometric;

\item \label{cond C2} $S_2$ is perfect, {\it i.e.} $\QC(S_2) \cong \Ind(\Perf(S_2))$;

\item \label{cond C3}  $S_2$ is Tannakian, {\it i.e.} for any $S \in \mathrm{dSt}$, one obtains a fully-faithfull map from $\Maps(S,S_2)$ into the category of symmetric monoidal k-linear functors from $\Perf(S_2)$ into $\Perf(S)$.
\end{enumerate}

Associated to our smooth projective curve $C$, consider its analytification $C^{\an}$. Porta \cite{porta} extended Simpson's construction of the Dolbeault and de Rham shapes to the analytic setting, giving rise to $(C^{\an})_{\dR}$, $(C^{\an})_{\B}$ and $(C^{\an})_{\Dol}$. After Holstein--Porta \cite[Lemmas 5.24, 5.27 and 5.31]{holstein&porta}, these constructions commute with analytification, {\it i.e.}
\[
(C^{\an})_{\dR} \cong (C_{\dR})^{\an}, \quad(C^{\an})_{\B} \cong (C_{\B})^{\an} \quad \text{and} \quad (C^{\an})_{\Dol} \cong (C_{\Dol})^{\an},
\]
and one can simply write $C^{\an}_{\dR}$, $C^{\an}_{\B}$ and $C^{\an}_{\Dol}$ for short. Furthermore, in the terminology of \textit{loc. cit}, one can check that $C_{\dR}$, $C_{\B}$ and $C_{\Dol}$ satisfy the universal GAGA property \cite[Propositions 5.26, 5.28 and 5.32]{holstein&porta}.  

Porta demonstrated in \cite{porta} that there exists a remarkable morphism between the analytic de Rham shape and the Betti one,
\begin{equation} \label{eq RH shape morphism}
\eta : C^{\an}_{\dR} \to C^{\an}_{\B}.
\end{equation}
It follows from Holstein--Porta \cite[Corollary 7.6]{holstein&porta} that, whenever $S_2 \in \mathrm{dSt}$ satisfies conditions \ref{cond C1}, \ref{cond C2} and \ref{cond C3} above, there exists an equivalence of derived analytic stacks
\[
\anMaps(C_{\dR}, S_2^{\an}) \cong \anMaps(C_{\B}, S_2^{\an}),
\]
which is induced by pull-back under $\eta$. 

\begin{remark} 
Obviously, one can not invert the morphism \eqref{eq RH shape morphism}, so the construction of a hypothetical quasi-algebraic Deligne shape is infeasible as any morphism $C_{\Betti} \times (\AA^1 - 0) \to C_{\Hod}$ factors necessarily through the $0$-section $C$.
\end{remark} 

In the remainder of this section, we identify a sufficiently broad class of stacks for which conditions \ref{cond C1}, \ref{cond C2}, and \ref{cond C3} are satisfied.

Consider an algebraic group $G$ with analytification $G^{\an}$ and a stack $X \in \mathrm{dSt}$ equipped with a $G$-action. Naturally, its analytification $X^{\an} \in \mathrm{dAnSt}$ inherits a $G^{\an}$-action. The preservation of colimits and finite limits under $(\cdot)^{\an}$ allows the identification 
\[
[X^{\an}/G^{\an}] \cong [X/G]^{\an},
\]
between the quotient stack of the analytification and the analytifications of the quotient stack.

\begin{lemma} \label{lm conditions for X}
Let $X$ be a derived quasiprojective scheme with affine diagonal equipped with an action by a complex reductive Lie group $G$. Then, the quotient stack $[X/G]$ satisfies \ref{cond C1}, \ref{cond C2} and \ref{cond C3}.
\end{lemma}

\begin{proof}
As $X$ is $0$-geometric, then $[X/G]$ is automatically $1$-geometric and satisfies \ref{cond C1}. It is proven in \cite[Corollary 3.22]{benZvi&francis&nadler} that $[X/G]$ satisfies \ref{cond C2}. 

Finally we address condition \ref{cond C3}. As a consequence of \cite[Theorem 5.11]{lurie_tannakian} (see also \cite{lurie_VIII, bhatt&HalpernLeistner}) a quasi-compact derived stack with affine diagonal is Tannakian. Note that $[X/G]$ is automatically quasi-compact as its atlas $X$ is quasi-projective by hypothesis. Furthermore, $[X/G]$ has affine diagonal as so does $X$ and $G$ is affine.
\end{proof}

When $Z$ is equipped with two commuting actions of the complex reductive Lie groups $G$ and $G'$, the successive quotient stack $[[Z/G']/G]$ amounts to a quotient stack $[Z / G' \times G]$. Then, one can intermediately extend the category of spaces that satisfy the desired conditions.

\begin{corollary}\label{co conditions for [X/H]}
Let $Z$ be a quasiprojective derived scheme with affine diagonal acted by the complex reductive Lie groups $G$ and $G'$ with commuting actions. Set $X := [Z/G']$ which is naturally equipped with a $G$-action. Then, $[X/G]$ satisfies \ref{cond C1}, \ref{cond C2} and \ref{cond C3}.
\end{corollary}

Under the hypothesis of Corollary \ref{co conditions for [X/H]}, all of the above allow leads to the identifications
\[
\Loc_{G}^{X}(C)^{\an} \cong \anMaps(C^{\an}_{\dR}, [X^{\an}/G^{\an}]),
\]
\[
\Rep_{G}^{X}(C)^{\an} \cong \anMaps(C^{\an}_{\B}, [X^{\an}/G^{\an}]),
\]
\[
\Higgs_{G}^{X}(C)^{\an} \cong \anMaps(C^{\an}_{\Dol}, [X^{\an}/G^{\an}]),
\]
and to the existence of the (Riemann--Hilbert) equivalences of derived analytic stacks
\begin{equation}\label{eq Betti cong DRahm}
\check{\eta} : \Loc_G^X(C)^{\an} \stackrel{\cong}{\to} \Rep_G^X(C)^{\an}.
\end{equation}

\begin{remark} \label{rm Betti cong DRahm moduli space}
It is well known that, in the case of $X = pt$, the identification \eqref{eq Betti cong DRahm} restricts to the semistable loci and descends to a complex analytic identification of the corresponding moduli spaces,
\[
\check{\eta}_0 : \Mm_{\dR}(C,G)^{\an} \stackrel{\cong}{\to} \Mm_{\B}(C,G)^{\an}.
\]
\end{remark}

\subsection{The relative Deligne stack}

For a reductive group $G$, the twistor space of the moduli space of semistable Higgs bundles $\Mm_{\Dol}(C,G)$ was described by Simpson in \cite{simpson_hodge_2} following Deligne. We can understand this construction using the notion of quasi-algebraic stacks (spaces). Starting from the Hodge moduli spaces 
\[
\Mm_{\Hod}(C,G) \to \AA^1, \qquad \text{and} \qquad \Mm_{\Hod}(\ol C,G) \to \AA^1,
\]
and their analytifications, we compose the Riemann-Hilbert isomorphism between the Betti and de Rahm moduli spaces of Remark \ref{rm Betti cong DRahm moduli space} with the inverse of the (analytification of the) trivialization induced from \eqref{eq algebraic Triv map}. By doing so, one obtains the pair of immersions $\ol{\mu}_0$ and $\mu_0$ described, respectively, as the top and botton rows of 
\begin{equation}
\label{eq Deligne construction}
\begin{tikzcd}
\Mm_{\B}(\ol{C}^{\top},G)^{\an} \times (\AA^1 - 0) \quad
\arrow[rr, "\Triv^{\an, -1} \circ (\check{\eta}_0^{-1} \times \id )"]
& &
\quad \Mm_{\Hod}(\ol C,G)^{\an} \times_{\AA^1} (\AA^1 - 0)
\arrow[r, hook]
&
\Mm_{\Hod}(\ol C,G)^{\an}
\\
\Mm_{\B}(C^{\top},G)^{\an} \times (\AA^1 - 0) \quad
\arrow[rr, "\Triv^{\an, -1} \circ (\check{\eta}_0^{-1} \times \id )"]
\arrow[u, equal, "\mathrm{id} \times (\lambda \leftrightarrow -\lambda^{-1})"]
& &
\quad \Mm_{\Hod}(C,G)^{\an} \times_{\AA^1} (\AA^1 - 0)
\arrow[r, hook]
&
\Mm_{\Hod}(C,G)^{\an}
\end{tikzcd},
\end{equation}
where we note that an analytic curve $\Cc$ and its complex conjugate share the same underlying topological space, $C^{\top} = \ol C^{\top}$ up to a reversal in orientation. Packing the data of \eqref{eq Deligne construction} into the quasi-algebraic space, we define the Deligne moduli space
\[
\Mm_{\Del}(C,G) := \left ( \Mm_{\B}(C^{\top},G)^{\an} \times (\AA^1 - 0), \Mm_{\Hod}(C,G) \sqcup \Mm_{\Hod}(\ol C,G), \mu_0, \ol\mu_0 \right ),
\]
which comes naturally equipped with the (quasi-algebraic, via Example \ref{ex PP^1 qalg}) structural projection
\begin{equation} \label{eq structural morphism Mm_Del}
\Mm_{\Del}(C,G) \to \P^1.
\end{equation}
The analytification of the Deligne moduili space recovers the Deligne--Simpson construction of the twistor space of $\Mm_{\Dol}(C,G)$,
\begin{equation} \label{eq M_Del is tw M_Dol}
\Mm_{\Del}(C,G)^{\an} := \Tw\left ( \Mm_{\Dol}(C,G) \right).
\end{equation}

In order to enhance the twistor construction to relative moduli stacks we make use of the notation established in the previous sections. Recall that the Betti shape of an analytic curve depends only on its topology so one can naturally identify the Betti shapes of a curve and that of its complex conjugate,
\[
C_{\B} = \ol{C}_{\, \B}.
\]
Note that the identification above comes with \textit{an inversion on orientation}, along which we identify their corresponding mapping stacks,
\begin{equation} \label{eq Betti for C an olC}
\Rep_{G}^{X}(C) = \Rep_{G}^{X}(\ol{C}).
\end{equation}

Let $X=[Z/G']$ be the quotient stack of a quasi-projective derived scheme $Z$ with affine diagonal by the action of the complex reductive group $G'$. Suppose that $X$ is further equipped with an action by the reductive group $G$ and recall that Corollary \ref{co conditions for [X/H]} ensures that $[X/G]$ satisfies conditions \ref{cond C1}, \ref{cond C2} and \ref{cond C3}. We now address a construction, analogous to that of $\Mm_{\Del}(C,G)$, involving the relative moduli stacks described in Section \ref{sc simpson shapes and mapping stacks}. Recall from \eqref{eq algebraic Triv map} the equivalence $\Triv$ and consider the inverse of its analytification. One can also invert $\check{\eta}$ in \eqref{eq Betti cong DRahm}, and, by composing the previous morphisms one then obtains the pair of immersions
\begin{equation}
\label{eq definition of mu and olmu}
\begin{tikzcd}
\Rep_{G}^{X}(C)^{\an} \times (\AA^1 - 0) \quad
\arrow[rr, "\Triv^{\an, -1} \circ (\check{\eta}^{-1} \times \id )"]
& &
\quad \Hodge_{G}^{X}(\ol{C})^{\an} \times_{\AA^1} (\AA^1 - 0)
\arrow[r, hook]
&
\Hodge_{G}^{X}(\ol{C})^{\an}
\\
\Rep_{G}^{X}(C)^{\an} \times (\AA^1 - 0) \quad
\arrow[rr, "\Triv^{\an, -1} \circ (\check{\eta}^{-1} \times \id )"]
\arrow[u, equal, "\mathrm{id} \times (\lambda \leftrightarrow -\lambda^{-1})"]
& &
\quad \Hodge_{G}^{X}(C)^{\an} \times_{\AA^1} (\AA^1 - 0)
\arrow[r, hook]
&
\Hodge_{G}^{X}(C)^{\an}
\end{tikzcd},
\end{equation}
which we denote by $\mu$ in the bottom row, and $\ol{\mu}$, in the top row. Observe that \eqref{eq Betti for C an olC} also makes part of the latter. Inspired by \cite{franco&hanson}, we define the {\it relative Deligne stack} for the $G$-scheme $X$ to be the derived quasi-algebraic stack
\[
\Deligne_{G}^{X}(C) := \left ( \Rep_{G}^{X}(C)^{\an} \times (\AA^1 - 0), \Hodge_{G}^{X}(C), \Hodge_{G}^{X}(\ol C), \mu, \ol \mu  \right ).
\]

\begin{remark}
Note that different $G$-actions on a fixed $X$ give rise, in general, to different relative Deligne moduli stacks. Nevertheless, we stick to this notation for simplicity. 
\end{remark}

\begin{remark} \label{rm pivotal morphism}
For $X$ as before, post-composition with $[X/G] \to BG$ produces a series of algebraic and analytic morphisms that glue to provide the quasi-algebraic morphism
\begin{equation} \label{eq pivotal morphism}
\theta^X : \Del^X_G(C) \to \Del_G(C).
\end{equation}
As we shall see in the following sections, the latter is pivotal in our constructions.
\end{remark}

Observe that its analytification, $\Deligne_{G}^{X}(C)^{\an}$, is a derived analytic stack corresponding to the pushout
\begin{equation}
\label{eq define glu}
\begin{tikzcd}[column sep = huge]
\Rep_{G}^{X}(C)^{\an} \times (\AA^1 - 0)
& \Hodge_{G}^{X}(\ol{C})^{\an}    
\\
\Hodge_{G}^{X}(C)^{\an}  &
\Deligne_{G}^{X}(C)^{\an}
\arrow[from = 1-1, to = 2-2, phantom, "\square"] 
\arrow[from = 1-1, to = 1-2, "\ol{\mu}", hook] 
\arrow[from = 1-1, to = 2-1, "\mu"', hook] 
\arrow[from = 1-2, to = 2-2, "\jmath_{\ol{\Hod}}"]
\arrow[from = 2-1, to = 2-2, "\jmath_{\Hod}"']
\end{tikzcd}.
\end{equation}
The structural morphisms $\Hodge_{G}^{X}(C) \to \AA^1$ and $\Hodge_{G}^{X}(\ol C) \to \AA^1$ induce the quasi-algebraic morphism
\begin{equation}
\label{eq define tau}
\tau : \Deligne_G^X(C) \to \P^1,
\end{equation}
where $\Hodge_{G}^{X}(C)$ and $\Hodge_{G}^{X}(\ol{C})$ are, respectively, the preimages of $\AA^1_{\Hod}, \AA^1_{\ol\Hod} \subset \P^1$. Moreover, $\Higgs_{G}^{X}(C)$, $\Loc_{G}^{X}(C)$ and $\Higgs_{G}^{X}(\ol C)$ are the fibers over $0, 1,\infty$, respectively. Obviously, the previous definition specialises to the Deligne stack when $X = pt$ which was previously constructed in \cite{franco&hanson}.

We say that a quasi-algebraic morphism $\zeta : Z \to \Mm$ is a {\it good moduli space} for the quasi-algebraic (derived) stack $Z$ if $\Mm$ is a quasi-algebraic space and $\zeta_*$ is an exact functor for ind-coherent sheaves inducing an isomorphism $\omega_{\Mm} \cong \zeta_*\omega_Z$. 

A statement similar to the one below appears in \cite{franco&hanson} for the Deligne moduli stack. Denote by $\Deligne_{G}^{\sst}(C)$ and $\Deligne_{G}^{\st}(C)$ the quasi-algebraic stacks obtained by restricting, respectively, to the semistable and stable loci of the Hodge moduli stacks.

\begin{proposition} \label{pr Del and Mm_Del}
\label{pr deligne moduli}
Let $G$ be a complex reductive group. The Deligne moduli space is a good moduli space for the classical semistable locus of the Deligne stack,
\begin{equation} \label{eq zeta_Del}
\zeta_{\Del} : t_0 \left ( \Deligne_{G}^{\sst}(C) \right ) \to \Mm_{\Del}(C,G).
\end{equation}
Furthermore, the truncation of \eqref{eq define tau} amounts to the composition of \eqref{eq zeta_Del} and the structural morphism $\Tw(\Mm_{\Dol}(C,G) \to \P^1$.

Finally, the restriction to the stable locus,
\begin{equation} \label{eq gerbe structure of Deligne^st}
\zeta_{\Del} : t_0 \left ( \Deligne_{G}^{\st}(C) \right ) \to \Mm_{\Del}^{\st}(C,G)
\end{equation}
is a $Z_G$-gerbe, where $Z_G$ denotes the centre of $G$.
\end{proposition}

\begin{proof}
First of all, we construct the morphism of derived quasi-algebraic stacks \eqref{eq zeta_Del} setting
\[
\zeta_{\Del} := \left ( \zeta_{\B} \times \id_{(\AA^1 - 0)}, \zeta_{\Hod}, \zeta_{\ol\Hod}  \right ),
\]
we obtain the quasi-algebraic morphism \eqref{eq zeta_Del}. As $\zeta_{\Hod}$ and $\zeta_{\ol \Hod}$ are good moduli spaces, their push-forward are exact on quasi-coherent sheaves and so is $\zeta_{\Del,*}$. Finally, since 
\[
\Oo_{\Mm_{\Hod}} \cong \zeta_{\Hod,*}\Oo_{\Hod} \qquad \text{and} \qquad \Oo_{\Mm_{\ol \Hod}} \cong \zeta_{\ol \Hod,*}\Oo_{\ol \Hod},
\]
and they both restrict trivially to $\Oo_{\Mm_{\B}} \cong \zeta_{\B,*}\Oo_{\Rep \times (\AA^1 - \{0 \})}$, one has the identification 
\[
\Oo_{\Mm_{\Del}} \cong \zeta_{\Del,*}\Oo_{\Del},
\]
and $\zeta_{\Del}$ is indeed a good moduli space. 

The second statement is trivial from the construction of \eqref{eq define tau} and \eqref{eq zeta_Del}.

We know that the classical truncations of the Dolbeault, de Rham and Hodge stable moduli stacks are $Z_G$-gerbes over their respective good moduli spaces. The Deligne moduli stack is covered by the Hodge moduli stacks over $C$ and $\ol C$, so the third statement follows.
\end{proof}

\begin{remark} \label{rm Del is split}
Being mapping stacks where both the target and the source are classical stacks, observe that $\Hodge_{G}(C)$, $\Hodge_{G}(\ol C)$ and $\Rep_{G}(C)$ are split derived stacks ({\emph i.e.} equipped with a retraction of the inclusion of their classical truncations) due to the intrinsic properties of mapping stacks. This produces a morphism of quasi-algebraic stacks 
\[
\rho : \Deligne_{G}^{X}(C) \to t_0 \left ( \Deligne_{G}^{X}(C) \right )
\]
which $\rho$ is a retraction of the inclusion of the classical truncation.
\end{remark}

\begin{remark} \label{rm real form on Deligne}
Any real form $K$ on the group $G$ induces naturally a real structure on $\Rep_G(C)$, which naturally extends to 
\[
\chi^K_{\B} : \Rep_G(G) \to \Rep_G(C), \quad (\rho, \lambda) \mapsto (\ol \rho^K, -\ol\lambda^{-1} ).
\]
The above extends to 
\[
\chi^K_{\Hod} : \Hodge_{G}(C) \to \Hodge_G(\ol C),
\]
and one can further construct a real structure 
\[
\chi^K : \Del_{G}(C) \to \Del_{G}(C)
\]
covering the antipodal map $\chi_{\P^1}$, as constructed in \eqref{eq chi P^1}. This picture descends naturally to the moduli space and Simpson showed \cite{simpson_hodge} that the real structure $\chi^{\RR}$ associated to the split real form $G^\RR$ is compatible with non-abelian Hodge theory. For the study of real structures on the Deligne moduli space associated to other real forms see \cite{biswas&heller&roser}. 
\end{remark}

\subsection{The case of the additive group}

The goal of this section is to provide an explicit description of the construction of the Deligne moduli stack for the additive abelian group $\GG_a$, which will be crucial in Section \ref{section TwB}. Essentially, one is studying the (Betti, Hodge, de Rham) cohomology of the base curve, interpolated by usual Hodge theory.

Given the Hodge shape $C_{\Hod} \stackrel{\pi}{\to} \AA^1$ of a smooth projective curve $C$, the trivialization morphism \eqref{eq Hodge specialises to de Rahm} induces the following isomorphism once we restrict to $(\AA^1 - \{ 0 \}) \stackrel{\jmath}{\hookrightarrow} \AA^1$ and make use of the Riemann-Hilbert correspondence
\[
\triv^n : \jmath^!R^n\Gamma(C_{\Hod}) \stackrel{\cong}{\to} H^n(C_{\dR}) \otimes_k \Oo_{(\AA^1 - 0)}.
\]
With the above, we construct a sheaf that captures the behavior of the cohomology of the (Simpson shape of the) curve along $\P^1$. 

Consider $\mu_1 : R^n\Gamma(C_{\B})^{\an} \boxtimes \Oo_{\GG_m}^{\an} \to R^n\Gamma(C_{\Hod})$ obtained as the composition of the pull-back under $\jmath$, the inverse of $\triv^n$, the isomorphism $H^n(C_{\dR}) \cong H^n(C_{\B})$ induced by the Riemann--Hilbert correspondence and the analytification functor. Define analogously $\mu_2$ making use of the corresponding morphisms for $\ol{C}$. Let us consider the associated homotopy equalizer,
\[
\C := \mathrm{holim} \left ( \begin{tikzcd}
	R^n\Gamma(C_{\B})^{\an} \boxtimes \Oo_{\GG_m}^{\an}   &  R^n\Gamma(C_{\Hod}) \sqcup R^n\Gamma(\ol{C}_{\Hod})
	\arrow["{\mu_1}"', shift right, from=1-2, to=1-1]
	\arrow["{\mu_2}", shift left, from=1-2, to=1-1]
\end{tikzcd} \right ).
\]
This is, by construction, a quasi-algebraic complex over $\P^1$,
\[
\C \in \QA^!(\P^1).
\]
We now describe the cohomology of $\C$.

\begin{lemma} \label{lm cohomology of the curve}
For $n \in \{ 0, 1, 2 \}$, 
\[
H^n(\C) \cong H^n(C_{\B}) \otimes_k \Oo_{\P^1}(n),
\]
and 
\[
H^n(\C) = 0,
\]
for all other $n$.
\end{lemma}

\begin{proof}
The Hodge cohomology sheaves trivialize along $\AA^1_{\Hod}$ and $\AA^1_{\ol \Hod}$. By dimensional considerations,
\[
R^n\Gamma(C_{\Hod}) \cong H^n(C_{\B}) \otimes_k \Oo_{\AA^1}, \quad R^n\Gamma(\ol{C}_{\Hod}) \cong H^n(C_{\B}) \otimes_k \Oo_{\AA^1}.
\]

It follows from \cite{lurie_ff_curve}, the key observation that, under $\triv^n$, the usual $\GG_m$ action on $C_{\Hod}$ induces to a $\GG_m$-action on $H^n(C_{\dR}) \times (\AA^1 - \{ 0 \} )$ with weights $(n,1)$. Hence, the result follows immediately as $\P^1$ is constructed from $\AA^1_{\Hod} \sqcup \AA^1_{\ol\Hod}$ and the gluing morphisms $\mu_1$ of weight $0$ and $\mu_2$ of weight $-1$. 
\end{proof}

Observe that the isomorphism described in Lemma \ref{lm cohomology of the curve} is not canonical. In the particular case of $n = 2$, it depends on the choice of isomorphisms 
\[
[C]_{\Sim}: R^2\Gamma(C_{\Sim}) \stackrel{\cong}{\to} \AA^1,
\]
for $\Sim \in \{\dR, \B, \Dol\}$, as well as  
\[
[C]_{\Hod} : R^2\Gamma(C_{\Sim}) \stackrel{\cong}{\to} \Oo_{\AA^1}, 
\]
or a $\Sim$-{\it orientation} of the curve $C$, for the corresponding choice of Simpson shape. Following Lemma \ref{lm cohomology of the curve}, the twistor version of such a notion is evident. 

\begin{definition}
    A $\P^1$-\textit{orientation} of the curve $C$ is an identification
    \begin{equation}
        [C]_{\P^1}: H^2(\Cc) \stackrel{\cong}{\to} \Oo_{\P^1}(2).
    \end{equation}
\end{definition}

\begin{remark}
In the Betti case, the following diagram commutes by construction, 
\[    
\begin{tikzcd}
	R^2\Gamma(C_{\Betti}) & R^2\Gamma(\ol{C}_{\Betti}) \\
	\AA^1 & \AA^1
	\arrow["{[C]_{\B}}"', from=1-1, to=2-1]
	\arrow["{\ol{[C]}_{\B}}", from=1-2, to=2-2]
	\arrow["{z \mapsto -z}", from=2-1, to=2-2]
\end{tikzcd}.
\]
Then, one immediately observe that a choice of compatible Hodge and Betti orientations for $C$ and $\ol{C}$ naturally induces a choice of $\P^1$-orientation.
\end{remark}

The description of the ($\Sim$)-cohomology of the curve allow us to describe the tangent complex of the Deligne moduli stack for the multiplicative group.

\begin{lemma} \label{lm description of T_Del}
The tangent complex $\TT_{\Del_{\GG_m}}$ of the Deligne moduli stack for $\GG_m$ has cohomology supported in degrees $-1$, $0$ and $1$, with
\begin{equation} \label{eq description of T_Del}
H^{-1}(\TT_{\Deligne_{\GG_m}(C)}) \cong \tau^*\Oo_{\P^1} \quad \text{and} \quad H^{1}(\TT_{\Deligne_{\GG_m}(C)}) \cong \tau^*\Oo_{\P^1}(2),
\end{equation}
and
\[
H^{0}(\TT_{\Deligne_{\GG_m}(C)}) \cong \tau^*\Oo_{\P^1}(1)^{\oplus 2g} \oplus \tau^*\Oo_{\P^1}(2).
\]
\end{lemma}

\begin{proof}
By construction, $\TT_{\Deligne_{\GG_m}}$ is the homotopy equalizer
\[
\TT_{\Deligne_{\GG_m}} = \mathrm{holim} \left ( \begin{tikzcd}
	\TT_{\Rep^{\an}_{\GG_m} \times \GG_m}   &  \TT_{\Hod_{\GG_m}} \sqcup \TT_{\ol \Hod_{\GG_m}}
	\arrow[shift right, from=1-2, to=1-1]
	\arrow[shift left, from=1-2, to=1-1]
\end{tikzcd} \right ) ,
\]
where the two arrows correspond to $\mu^! \circ (\cdot)^{\an}$ and $\ol\mu^! \circ (\cdot)^{\an}$. Since $\GG_m$ is abelian, the universal $\GG_m$-bundle $pt \to B\GG_m$ induces a trivial adjoint vector bundle. Then, $\TT_{\Sim_{\GG_m}}$ and $\TT_{\Hod_{\GG_m}/\AA^1}$ are provided by the pull-back of the cohomology sheaf of the corresponding Simpson shape of the curve (shifted by $1$). In particular, 
\[
H^n(\TT_{\Hod_{\GG_m}/\AA^1}) \cong \tau_{\Hod}^* R^{n+1}\Gamma(C_{\Hod}), \quad H^n(\TT_{\ol\Hod_{\GG_m}/\AA^1}) \cong \tau_{\ol \Hod}^* R^{n+1}\Gamma(C_{\ol \Hod}),
\]
and
\[
H^n(\TT_{\Rep_{\GG_m}}) \cong \tau_{\B}^* R^{n+1}\Gamma(C_{\B}).
\]
Recalling that $\TT_{\P^1}$ is fully supported in degree $0$, by smoothness, one has for every $n \neq 0$
\[
H^n(\TT_{\Hod_{\GG_m}/\AA^1}) \cong H^n(\TT_{\Hod_{\GG_m}}), \quad H^n(\TT_{\ol\Hod_{\GG_m}/\AA^1}) \cong H^n(\TT_{\ol\Hod_{\GG_m}}), 
\] 
and 
\[
H^n(\TT_{\Rep_{\GG_m} \times \GG_m}) \cong H^n(\TT_{\Rep_{\GG_m}}),
\]
while
\[
H^0(\TT_{\Hod_{\GG_m}}) \cong \tau_{\Hod}^* R^{1}\Gamma(C_{\Hod}) \boxplus \Oo_{\AA^1}, \quad H^0(\TT_{\ol\Hod_{\GG_m}}) \cong \tau_{\ol \Hod}^* R^{1}\Gamma(C_{\ol \Hod}) \boxplus \Oo_{\AA^1},
\]
and
\[
H^n(\TT_{\Rep_{\GG_m}}) \cong \tau_{\B}^* R^{n+1}\Gamma(C_{\B}) \boxplus \Oo_{(\AA^1- 0)},
\]
noting that the extra terms carry a weight $1$ action of $\GG_m$.

Hence, the proof follows immediately from Lemma \ref{lm cohomology of the curve}.
\end{proof}

\begin{remark} \label{rm Dirac for ad}
Note that the relative tangent complex $\TT_{\Del_{\GG_m}/ \P^1}$ has cohomology in degrees $-1$, $0$ and $1$, with
\[
H^{-1}(\TT_{\Deligne_{\GG_m}/ \P^1}) \cong \tau^*\Oo_{\P^1} \quad \text{and} \quad H^{1}(\TT_{\Deligne_{\GG_m}/ \P^1}) \cong \tau^*\Oo_{\P^1}(2),
\]
and
\[
H^{0}(\TT_{\Deligne_{\GG_m}/ \P^1}) \cong \tau^*\Oo_{\P^1}(1)^{\oplus 2g}.
\]
\end{remark}

Since $\GG_m$ has Lie algebra $\GG_a$, the description of the Deligne moduli stack for the latter follows from the previous statements.

\begin{proposition} \label{pr Del_Ga}
There is an isomorphisms of derived stacks
\[
\Deligne_{\GG_a}(C) \cong B\GG_a \times \left ( \OO(2) \times_{\P^1} \OO(1)^{\times_{\P^1} 2g} \right ) \times_{\P^1} \OO(2)[-1].
\]
\end{proposition}

\begin{proof}
Consider the trivial section
\[
\sigma_0 : \P^1 \to \Del_{\GG_m}(C)
\]
obtained by considering the trivial bundles $\Oo_{C_{\Hod}}$ and $\Oo_{\ol C_{\Hod}}$. Using Lemma \ref{lm description of T_Del}, one can easily describe the restriction of the tangent complex to the trivial section,
\begin{equation} \label{eq TT_Del on sigma0}
\sigma_0^* \TT_{\Del_{\GG_m}} \cong \Oo_{\P^1}[1] \oplus \left ( \Oo_{\P^1}(1)^{\oplus 2g} \oplus \Oo_{\P^1}(2) \right ) \oplus \Oo_{\P^1}(2)[-1]. 
\end{equation}
Since $\GG_a$ is the Lie algebra of $\GG_m$, it follows that the $\GG_a$-Deligne moduli stack amounts to the total space of the restriction to the trivial section of the tangent complex of $\Deligne_{\GG_m}(C)$. Then, the proof follows from \eqref{eq TT_Del on sigma0}.
\end{proof}

It is clear from Proposition \ref{pr Del_Ga} that $\Del_{\GG_a}(C)$ project to its (strictly) derived component $\OO(2)[-1]$. Nevertheless, as in the case of Lemma \ref{lm cohomology of the curve}, such projection is not canonical. This justifies the following. 

\begin{definition} \label{df Del orientation}
    A \textit{Deligne orientation} of the curve $C$ is a surjective morphism over $\P^1$,
    \begin{equation}
        [C]_{\Del}: \Del_{\GG_a}(C) \to  \OO(2)[-1].
    \end{equation}
\end{definition}

\begin{remark}
It follows naturally from the proof of Proposition \ref{pr Del_Ga} that a choice of Deligne orientation depends on the choice of a $\P^1$-orientation (hence on the choices of compatible Betti and Hodge orientations for $C$ and its conjugate).
\end{remark}

\section{Hamiltonian actions and boundary conditions}
\label{section TwB}

One of the guiding ideas of Ben-Zvi--Sakellaridis--Venkatesh \cite{BZSV}, going back to Gaiotto and Witten's work \cite{gaiotto&witten, Gaiotto} , is the emphasis on Hamiltonian $G$-spaces as the source of boundary conditions for the supersymmetric 4d Yang–Mills theory studied by Kapustin--Witten \cite{kapustin&witten}.  At a physical level of rigor, this means that there exists a Moore--Tachikawa category of boundary conditions which controls the Langlands QFT in the spirit of the Cobordism Hypothesis, {\it i.e.} that the Langlands QFT is encoded by a representation of the Moore--Tachikawa.

In this section, we explore the construction of a functor on the \textit{category of boundary conditions for the B-twist Langlands TQFT}\footnote{Here, by ``topological" we really mean ``conformal", i.e., the resulting structures continue to depend on the complex structure of our chosen algebraic curve. In the following, we shall use ``QFT" to mean that we have not picked a topological twist, and ``TQFT" to mean that we are considering either the A or B-twist. }. Our main result Theorem \ref{th representation of Cc} fits into the first steps of such a program, while Corollaries \ref{co relation with BZSV} and \ref{co relation with Gaiotto} exhibits the connection of our construction with the work of \cite{BZSV} and \cite{Gaiotto}, demonstrating that they interpolate previous results on this topic. 

\subsection{Quantization of the Moore--Tachikawa category}

\subsubsection{The Moore--Tachikawa category}

Mathematically, the Moore--Tachikawa category \cite{Moore--Tachikawa} may be conceived as a \textit{1-shifted Weinstein category} \cite{Calaque, Crooks-Mayrand, Weinstein1, Weinstein2}, whose objects are $1$-shifted symplectic stacks of the form $T^*[1]BG$ for reductive groups $G$, and 1-morphisms are furnished by 1-Lagrangian correspondences. Given two comparable 1-morphisms, i.e., a pair of 1-Lagrangian correspondences with the same source and target, their fiber product naturally carries the structure of a 0-shifted symplectic stack. The 2-morphisms of the Moore--Tachikawa category can then defined as 0-shifted Lagrangians in this fiber product, and one can continue on in this manner to build an $\infty$-category whose $k$-morphisms are furnished by $(2-k)$-Lagrangians. 

However, we will not consider these structures for $k>2$ and we explicate only the first three layers of this category. For that, the key mathematical definition we shall need is the following. 
\begin{definition}
A graded Hamiltonian $G$-action, is a symplectic $G$-variety $M$ equipped with a $G \times \Ggr$-equivariant moment map $\mu: M \to \mathfrak{g}^*$ with $\Ggr$ scaling $\mathfrak{g}^*$ with weight 2. 
\end{definition}

With the notion of graded Hamiltonian action established, we may introduce the main characters of the Moore--Tachikawa category:
\begin{itemize}
    \item (Codimension 0). The objects are reductive groups $G$, or, equivalently, their $1$-shifted symplectic stacks of the form $T^*[1]BG$. 
    \item (Codimension 1). The ($1$-)morphisms from one reductive group $H$ to another reductive group $G$ are graded Hamiltonian (commuting) actions $H \acts M \racts G$, with composition provided by Hamiltonian reduction. 
    \item (Codimension 2). Let $M_1, M_2$ be two Hamiltonain $G$-actions, viewed as $1$-morphisms in the Moore--Tachikawa category (from the trivial group to $G$). The ($2$-)morphisms from $M_1$ to $M_2$ are $G \times \Ggr$-equivariant Lagrangian correspondences $z = z_1 \times z_2:Z \to M_1^- \times M_2$, where $M_1^-$ is the $G$-Hamiltonian space with underlying space equal to $M_1$ and whose symplectic form is negated. 
\end{itemize}

Mathematically, one can realize the Moore--Tachikawa category as a 1-shifted Weinstein category: the objects are ($\Ggr$-equivariant) 1-shifted symplectic stacks of the form $T^*[1]BG \simeq \mathfrak{g}^*/G$ (with weight 2 $\Ggr$-scaling action). The equivalence between 1-shifted Lagrangian correspondences between such 1-shifted symplectic stacks and Hamiltonian actions as explained by \cite{Safronov} recovers the preceding concrete presentation of this category. In codimension 2, we start with an intersection of 1-shifted Lagrangians 
$$[M_1^-/G] \times_{T^*[1]BG} [M_2/G] \simeq M_1^- \times_{\mathfrak{g}^*}^G M_2$$
which is equipped with 0-shifted symplectic structure, and equivariant Lagrangian correspondences correspond to 0-shifted Lagrangians in $M_1^- \times_{\mathfrak{g}^*}^G M_2$.

\subsubsection{Langlands duality and S-duality of boundary conditions}

Up until now, we have remained relatively agnostic to whether we are considering the \textit{A-twist} or the \textit{B-twist} of the Kapustin--Witten QFT. In both cases, we are considering a representation of the Moore--Tachikawa category into dg-Categories,
\[
A,B : \text{(Moore--Tachikawa)} \to \mathrm{dgCat}.
\]
S-duality predicts an equivalence bewteen these two representations compatible with Langlands duality in the following sense. One may envision a duality structure $\check{(-)}$ on the Moore--Tachikawa category of boundary conditions, which at the level of objects acts by
$$G \longmapsto \check{G} \text{ the Langlands dual group}.$$
It is with respect to this duality operation $\check{(-)}$ that S-duality is expected to intertwine A-twisting and B-twisting, in the sense that (for any algebraic curve $C$), one expects an equivalence 
$$A(G) \leftrightarrow B(\check{G}) \text{ and } B(G) \leftrightarrow A(\check{G}),$$
of dg-categories.

The work of Ben-Zvi--Sakellaridis--Venkatesh \cite{BZSV} sheds some light on the action of the purported duality $\check{(-)}$ at the level of 1-morphisms. Given a \textit{hyperspherical} Hamiltonian action $M$ one can consider its BZSV hyperspherical dual $\check{M}$. In this case, one expects the equivalence
$$A(G \acts M) \leftrightarrow B(\check{G} \acts \check{M}) \text{ and } B(G \acts M) \leftrightarrow A(\check{G} \acts \check{M})$$
via BZSV hyperspherical duality. Unpacking, these equations are identifications of \textit{objects} in the equivalences of categories exhibited in the preceding equivalence. These objects have been termed (BZSV) \textit{period sheaves} and \textit{L-sheaves} in the A and B-twists respectively, in reference to the role they play in the arithmetic Langlands program.

\begin{remark}
At the level of 2-morphisms some examples have been discovered in recent work of the first named author \cite{CHY} as well, although a general framework has not yet been understood. 
\end{remark}

\subsection{Spectral quantization of the polarized Moore--Tachikawa category}
\label{sc spectral quantization}

We address our main result in this section: the construction of a (B-twist) representation of an enhanced version of the Moore--Tachikawa category. The version of the Moore--Tachikawa category we use has its members equipped with extra data, akin to the \textit{prequantization data} required for geometric quantization. In the shifted symplectic setting, these notions were first laid out by \cite{Safronov}.

\subsubsection{The polarized Moore--Tachikawa category}
\label{sc polarized MT}

We are interested in constructing a version of the Moore--Tachikawa category, which we call $\Cc$, whose objects coincide with those of the Moore--Tachikawa category ({\it i.e.} complex reductive groups) although its morphisms are suitably enhanced with \textit{prequantization data} in preparation for the Langlands B-twist quantization. Concretely, this means that the 1-morphisms and 2-morphisms of $\Cc$ are equipped with the structure of certain \textit{Lagrangian fibrations} \cite{Safronov}. To set things up, let us start with the fundamental definition supplying 1-morphisms in $\Cc$. Recall that if $X$ is a $G$-variety, the canonical sequence of cotangent complexes associated to the quotient map $\pi: X \to X/G$ reads
\begin{equation}
    \pi^*\mathbf{L}_{X/G} \to \mathbf{L}_X \overset{\mu}{\to} \mathfrak{g}^* \otimes \mathcal{O}_X \overset{+1}{\to}
\end{equation}
where $\mu$ may be regarded as a $\mathfrak{g}^*$-valued, $G$-equivariant moment map from $T^*X \to \mathfrak{g}^*$. Moment maps that arise this way will be referred to as \textit{polarized}, as they are always equipped with the equivariant Lagrangian fibration structure $T^*X \to X$. The following definition gives a mild but important generalization. 

\begin{definition}
Given a $G \times \Ggr$-action $X$ with a $G \times \Ggr$-equivariant map $\alpha: X \to B\GG_a$ (with trivial $G$-action and weight $2$ $\Ggr$-action on the target), we may form the \textit{$\alpha$-twisted cotangent bundle} $T^*_\alpha X$, equipped with a $\mathfrak{g}^*$-valued $G$-equivariant moment map. We say that a graded Hamiltonian $G$-action $M$ is \textit{twisted polarized} if it arises from a pair $(X,\alpha)$, 
\[
M \cong T^*_{\alpha}X.
\]
\end{definition}

The above definition would be crucial in defining $1$-morphisms on the category $\Cc$. However, we need further adjustments, following Corollary \ref{co conditions for [X/H]}, that ensure that Conditions \ref{cond C1}, \ref{cond C2} and \ref{cond C3} are preserved.

\begin{definition}\label{defn 1-morphism}
$1$-morphisms of the category $\Cc$ are twisted polarized graded Hamiltonian $G$-actions $T_\alpha^*X$ where $X := [Z/G']$ being $Z$ a quasiprojective derived scheme with affine diagonal acted by the complex reductive Lie groups $G$ and $G'$ with commuting actions. The composition of the $1$-morphisms $M_{12} \in \Mor_{\Cc}(G_1, G_2)$ and $M_{23} \in \Mor_{\Cc}(G_2, G_3)$ is provided by the $G_2$-Hamiltonian reduction, 
\[
M_{23} \circ M_{12} :=  M_{12}^- \times_{\mathfrak{g}_2^*}^{G_2} M_{23} \,,
\]
which is naturally equipped with a $\Ggr$-action since the moment map is weight 2.
\end{definition}

As a preliminary step, we check the validity of Definition \ref{defn 1-morphism} showing that our proposed class of $1$-morphisms is closed under composition. For that, we first study the preservation of the twisted polarized structure.

\begin{lemma} \label{lemma composition of twisted polarization}
Given two twisted polarized graded $G$-Hamiltonian actions $M_{12} = T^*_\alpha X$ and $M_{23} = T^*_\beta Y$, we form the fiber product $M_{13} := M_{12}^- \times_{\mathfrak{g}_2^*} M_{23}$. Then, the Hamiltonian action $G_1 \times G_3 \acts M_{13}$ can be presented as $T^*_{\alpha \times \beta}[(X \times Y)/G_2]$. In particular, Hamiltonian reductions of twisted polarized Hamiltonian actions are twisted polarized.
\end{lemma}

\begin{proof}
This is a twisted version of the general fact that for any polarized Hamiltonian action $H \acts T^*W$, we have $W \times_{\mathfrak{h}^*}^H \{0\} \simeq T^*[W/H]$ (by taking $\alpha = \beta = 0, W = X \times Y$, and $H = G_2$). For the twisted statement, we would like to establish that, if $\alpha: W/G \to B\GG_a$ is a twist, then the Hamiltonian reduction of $T^*_\alpha W$ by $H$ is given by $T^*_\alpha[W/H]$, where in the latter we regard $\alpha: [W/H] \to B\GG_a$ as a twist for $T^*[W/H]$.

Taking the differential of the twisting data $\alpha$ we obtain a morphism of complexes
\begin{equation}
    d\alpha_z: \big[\mathfrak{h} \to T_x W\big][1] \longrightarrow \mathfrak{g}_a[1]
\end{equation}
for every point $x \in W$, and taking the degree $(-1)$ component we obtain $a := d\alpha_{(-1)}$, a $G$-equivariant $\mathfrak{h}^*$-valued function on $W$. The Hamiltonian moment map $\mu_\alpha$ for the twisted cotangent $T^*_\alpha W$ can be described in terms of $a$ as follows. Finding an $H$-equivariant trivialization $p: U \to W$ on which $p^*T^*_\alpha \simeq T^*U$, the fundamental sequence of cotangent complexes along the quotient $\pi: U \to U/H$ reads
\begin{equation} \label{equation moment map of polarized action}
    \pi^*\mathbf{L}_{U/H} \longrightarrow \mathbf{L}_U \overset{\mu}{\longrightarrow} \mathfrak{h}^* \otimes \mathcal{O}_U \overset{+1}{\to}
\end{equation}
where $\mu$ is the untwisted moment map $T^*U \to \mathfrak{h}^*$. Twisting $\mu$ by $p^*a$, that is, considering the formula
\begin{equation} \label{equation twisted moment map on a fiber}
    T^*_xU \ni \xi \longmapsto \mu(\xi) - a(x),
\end{equation}
we descend to a morphism $\mu_\alpha: T^*_\alpha W \to \mathfrak{h}^*$ which is the Hamiltonian moment map of $T^*_\alpha W$. Indeed, any change of local trivialization changes $\xi$ by an exact 1-form $\xi \mapsto \xi + df$, so the first summand of formula \eqref{equation twisted moment map on a fiber} gives, for a Lie algebra element $v \in \mathfrak{h}$ inducing a vector field that we continue to denote by $v$, 
\begin{equation}
    \mu(\xi+df)(v) = \langle \xi, v_x\rangle + v_x(f). 
\end{equation}
On the other hand, the function $a$ changes under the change of trivialization by
\begin{equation}
    a \mapsto a' = a+df, \text{ so } a'(x)(v) = a(x)(v) + \langle df, v\rangle_x = a(x)(v) + v_x(f),
\end{equation}
and the expression $\mu+a$ of \eqref{equation twisted moment map on a fiber} indeed descends.

To summarize, locally on $W$ we may model the twisted moment map $\mu_\alpha: T^*_\alpha W \to \mathfrak{h}^*$ by $\mu+a$. Therefore, by \eqref{equation moment map of polarized action} we may locally model $\mu_\alpha$-Hamiltonian reduction as $(\mu-a)$-Hamiltonian reduction, the latter gives precisely $T^*_\alpha [W/H]$ by definition.
\end{proof}

\begin{proposition}
The class of $1$-morphisms described in Definition \ref{defn 1-morphism} is closed under the proposed composition.
\end{proposition}

\begin{proof}
Making use of Lemma \ref{lemma composition of twisted polarization} and the notation therein, the remaining of the proof is immediate after the observation, proved below, that
\begin{equation} \label{eq comp of 1-mor with cond}
[(X \times Z)/G_2] \cong [Z/G],
\end{equation}
being $Z$ a quasiprojective derived scheme with affine diagonal and $G$ a complex reductive Lie group.

Suppose $X = [Z'/G']$ and $Y = [Z''/G'']$, with $Z'$ and $Z''$ quasiprojective derived schemes with affine diagonal acted by the complex reductive Lie groups $G'$ and $G''$ with actions commuting with $G_2$. Then, their product amounts to
\[
X \times Y = [Z'/G'] \times [Z''/G''] \cong [Z' \times Z''/G' \times G''],
\]
where $G'$ and $G''$ act trivially on the second and first factors, respectively, and $Z = Z' \times Z''$ is a quasi-projective derived scheme with affine diagonal as so are $Z'$ and $Z''$. 

As the action of $G' \times G''$ on $Z$ commutes with that of $G_2$ by hypothesis, \eqref{eq comp of 1-mor with cond} follows after setting $G = G' \times G'' \times G_2$ and the proof is completed.
\end{proof}

Next, we consider 2-morphisms; one could have phrased it in the generality of Lagrangian fibrations in the sense of \cite{Safronov}, but we will only apply these concepts to our restricted class of twisted polarized 1-morphisms. 
\begin{definition} \label{defn 2-morphism}
    Let $T^*_\alpha X, T^*_\beta Y$ be two 1-morphisms with the same source and target (say, with graded Hamiltonian action by $G$). A \textit{2-morphism} from $T^*_\alpha X$ to $T^*_\beta Y$ is a diagram
    \begin{equation}
        \begin{tikzcd}
	{T^*_\alpha X} & Z & {T^*_\beta Y} \\
	X & A & Y
	\arrow[from=1-1, to=2-1]
	\arrow["{q_1}"', from=1-2, to=1-1]
	\arrow["{q_2}", from=1-2, to=1-3]
	\arrow[from=1-2, to=2-2]
	\arrow[from=1-3, to=2-3]
	\arrow["{p_1}"', from=2-2, to=2-1]
	\arrow["{p_2}", from=2-2, to=2-3]
\end{tikzcd}
    \end{equation}
    where
    \begin{itemize}
        \item $q_1 \times q_2: Z \to T^*_\alpha X \times T^*_\beta Y$ is an equivariant Lagrangian correspondence (equivalently, $Z$ has a $G$-action and $Z/G \to T^*_\alpha X \times_{\g^*}^G T^*_\beta Y$ is a 0-Lagrangian),
        \item $Z \to A$ is a polarization, and
        \item $p_1$ is \'etale, 
    \end{itemize}
    equipped with an isomorphism $p_1^!\alpha \simeq p_2^!\beta$ of $\GG_a$-bundles on $A$.

    Consider the (composable) $2$-morphisms 
$$
\begin{tikzcd}
	{M_1 = T^*_{\alpha_1}X_1} & {Z_{12}} & {M_2 = T^*_{\alpha_2}X_2} & {M_2 = T^*_{\alpha_2}X_2} & {Z_{23}} & {M_3 = T^*_{\alpha_3}X_3} \\
	{X_1} & {A_{12}} & {X_2} & {X_2} & {A_{23}} & {X_3}
	\arrow[from=1-1, to=2-1]
	\arrow[from=1-2, to=1-1]
	\arrow[from=1-2, to=1-3]
	\arrow["{\pi_{12}}"', from=1-2, to=2-2]
	\arrow[from=1-3, to=2-3]
	\arrow[from=1-4, to=2-4]
	\arrow[from=1-5, to=1-4]
	\arrow[from=1-5, to=1-6]
	\arrow["{\pi_{23}}"', from=1-5, to=2-5]
	\arrow[from=1-6, to=2-6]
	\arrow["{p_1}"', from=2-2, to=2-1]
	\arrow["{p_2}", from=2-2, to=2-3]
	\arrow["{q_1}"', from=2-5, to=2-4]
	\arrow["{q_2}", from=2-5, to=2-6]
\end{tikzcd}
$$
where $p_1, q_1$ are \'etale, and $\pi_{12}, \pi_{23}$ are polarizations compatible with the twisted polarizations of $M_1, M_2, M_3$. We define their composition via the fibre products
\begin{equation} \label{eq fibre prods comp 2-mor}
Z_{13} := Z_{12} \times_{M_2} Z_{23} \to A_{13} := A_{12} \times_{X_2} A_{13}
\end{equation}
and the associated commuting diagram
\begin{equation} \label{eq commuting diagram for composition}
\begin{tikzcd}
	&&{Z_{13}} &  & &&& \\
    &&&&& {A_{13}} && \\
    & {Z_{12}} && {Z_{23}} &&& \\
	&&&& {A_{12}} && {A_{23}} \\
    {T^*_{\alpha_1}X_1} && {T^*_{\alpha_2}X_2} && {T^*_{\alpha_3}X_3} &&& \\
	&&&{X_1} && {X_2} && {X_3}
	\arrow["{r_1}"', from=2-6, to=4-5]
	\arrow["{r_2}", from=2-6, to=4-7]
	\arrow["{p_1}"', from=4-5, to=6-4]
	\arrow["{p_2}", from=4-5, to=6-6]
	\arrow["{q_1}"', from=4-7, to=6-6]
	\arrow["{q_2}", from=4-7, to=6-8]
    \arrow[from=1-3, to=3-2]
	\arrow[from=1-3, to=3-4]
	\arrow[from=3-2, to=5-1]
	\arrow[from=3-2, to=5-3]
	\arrow[from=3-4, to=5-3]
	\arrow[from=3-4, to=5-5]    
    \arrow[from=1-3, to=2-6]
	\arrow[from=3-2, to=4-5]
	\arrow[from=3-4, to=4-7]
	\arrow[from=5-1, to=6-4]
	\arrow[from=5-3, to=6-6]
	\arrow[from=5-5, to=6-8]
\end{tikzcd},
\end{equation}
and the associated isomorphism between the polarization data $(X_1, \alpha_1)$ and $(X_3, \alpha_3)$ obtained by 
$$r_1^!p_1^!\alpha_1 \simeq r_1^!p_2^!\alpha_2 \simeq r_2^!q_1^!\alpha_2 \simeq r_2^! q_2^! \alpha_3.$$
\end{definition}

\begin{proposition}
The class of $2$-morphisms described in Definition \ref{defn 2-morphism} is closed under the proposed composition.
\end{proposition}

\begin{proof}
This is immediate from the observation that \eqref{eq fibre prods comp 2-mor} forms a polarized Lagrangian correspondence \cite{Safronov} between $M_1$ and $M_3$ and the fact that the composition $p_1 \circ r_1$ is \'etale as so are $q_1$, by initial assumption, and $r_1$, by base change.
\end{proof}

Based on the above, we define the following polarized enhancement of the Moore--Tachikawa category.

\begin{definition}
Define the {\it polarized Moore--Tachikawa category} $\Cc$ as the $2$-category populated by the following data:
\begin{itemize}
    \item Objects in $\Cc$ are 1-shifted symplectic stacks of the form $T^*[1]BG$ for $G$ a reductive group.
    \item $1$-morphisms in $\Cc$ and their compositions are given as in Definition \ref{defn 1-morphism}.
    \item $2$-morphisms in $\Cc$ and their compositions are given as in Definition \ref{defn 2-morphism}.
\end{itemize}
\end{definition}

We are now ready to state our main result.

\begin{theorem} \label{th representation of Cc}
Given a smooth projective curve $C$, there exists a representation 
\[
\bB : \Cc \to \mathrm{dgCat},
\]
sending $G$ to the category of quasi-algebraic sheaves over the corresponding Deligne moduli stack,
\[
\bB(G) := \QA^!(\Del_G(C)).
\]
\end{theorem}

\begin{proof} 
The representation of $1$-morphisms and $2$-morphisms are covered in Sections \ref{sc 1-mor} and \ref{sc 2-mor}, respectively.
\end{proof}

\subsubsection{1-morphisms}
\label{sc 1-mor}

We construct, in this section, quasi-algebraic sheaves starting from twisted polarized Hamiltonian actions that are meant to be kernels for suitable functors between the corresponding categories. We abuse of notation and denote the functors and their kernels indistinctively.  

For shake of clarity, we first work out the case with trivial polarization ({\it i.e.} $\alpha = 0$). Suppose we have a graded polarized Hamiltonian $G$-action $M = T^*X$ with $X$ as in the statement of Corollary \ref{co conditions for [X/H]}, hence the existence of the associated relative Deligne stack and the morphism \eqref{eq pivotal morphism} is ensured,
\[
\theta^X : \Del_G^X(C) \to \Del_G(C).
\]
We define the quasi-algebraic sheaf attached to $M$ as follows. 
\begin{definition}
    Consider a quasiprojective derived scheme $Z$ with affine diagonal acted by the complex reductive Lie groups $G$ and $G'$ with commuting actions. Set $X := [Z/G']$ and $M = T^*X$, which is a polarized graded Hamiltonian $G$-action. Associated to it, we consider the quasi-algebraic sheaf
    \begin{equation}
        \bB(X) := \theta^X_{*}\omega \in \QA^!(\Del_G(C)).
    \end{equation}
\end{definition}
Specializing to the case when $G = G_1 \times G_2$ is a product group, we see that $B(X)$ can be used as an integral kernel along the correspondence 
\begin{equation}
    \begin{tikzcd}
	{\Del_{G_1}(C)} & {\Del_G(C)} & {\Del_{G_2}(C)}
	\arrow["{p_1}"', from=1-2, to=1-1]
	\arrow["{p_2}", from=1-2, to=1-3]
\end{tikzcd}
\end{equation}
to define the functor
\begin{equation} \label{equation spectral quantization of 1-morphisms}
    \bB(X):= p_{2,*}\big(p_1^!(-) \otimes^! \theta^X_{\Del,*}\omega \big): \QA^!(\Del_{G_1}(C)) \to \QA^!(\Del_{G_2}(C)),
\end{equation}
which represents the $1$-morphisms of $\Cc$ associated to $(X, \alpha = 0)$.

\begin{example}[Extension of structure group]
    Let $H \subset G$ be a reductive subgroup, and consider the action $H \acts X = G \racts G$. Then $\Del_{H \times G}^X(C) \to \Del_{H \times G}(C)$ is exactly the graph of the extension of structure group morphism $\Del_H(C) \to \Del_G(C)$, hence the functor $\bB(X)$ is just pushing forward in this case.
\end{example}

The construction in the twisted polarized case relies crucially on the consideration of the Deligne moduli stack of \textit{nonreductive groups}, which are necessarily derived (quasi-algebraic) stacks. 

Consider a twisted polarized Hamiltonian action $M = T^*_\alpha X$, where we regard $\alpha$ as a $G$-equivariant map $\alpha: X \to B\GG_a$ with trivial $G$-action on the target. Post-composition under $\alpha$ gives rise to algebraic and analytic morphisms between the corresponding Hodge and Betti moduli stacks. These morphisms glue to provide a quasi-algebraic morphism between the associated Deligne moduli stacks. We may consider the composition of the latter with a Deligne orientation, as defined in Definition \ref{df Del orientation}, to give $\Ggr$-equivariant quasi-algebraic morphisms
\begin{equation} \label{eq def bf alpha}
    \boldsymbol{\alpha}: \Del_G^X(C) \to \Del_{\GG_a}(C) \overset{[C]_{\Del}}{\to} \mathbf{O}(2)[-1]
\end{equation}
Recall the exponential sheaf $\mathbf{exp} \in \QA^!(\mathbf{O}(2)[-1])^\shear$, a sheared quasi-algebraic sheaf on $\mathbf{O}(2)[-1]$ constructed in Section \ref{sc qa exponential sheaf}. It will now play a crucial role in the description of the quasi-algebraic sheaf attached to $M$.

\begin{definition}
Consider a quasiprojective derived scheme $Z$ with affine diagonal acted by the complex reductive Lie groups $G$ and $G'$ with commuting actions. Set $X := [Z/G']$ and consider a $G \times \Ggr$-map $\alpha : X \to B \GG_a$. Let $M = T^*_\alpha X$ be the corresponding twisted polarized graded Hamiltonian $G$-action. The associated quasi-algebraic sheaf is 
    \begin{equation}
        \bB(X,\alpha) := (\theta_{*}^X\boldsymbol{\alpha}^!\mathbf{exp})^\unshear \in \QA^!(\Del_G(C)).
    \end{equation}
\end{definition}
To unpack slightly, we start in the sheared category $\boldsymbol{\alpha}^!\mathbf{exp} \in \QA^!(\Del_G^X(C))^\shear$, we pushforward via $\theta_{*}^X$ to land in $\QA^!(\Del_G(C))^\shear$, and then unshear $\unshear: \QA^!(\Del_G(C))^\shear \overset{\sim}{\to}\QA^!(\Del_G(C))$ via the trivial $\Ggr$-action on $\Del_{G}(C)$. 

Specializing to the case when $G = G_1 \times G_2$ is a product group, we see that $B(X,\alpha)$ can be used as an integral kernel along the correspondence
\begin{equation}
    \begin{tikzcd}
	{\Del_{G_1}(C)} & {\Del_G(C)} & {\Del_{G_2}(C)}
	\arrow["{p_1}"', from=1-2, to=1-1]
	\arrow["{p_2}", from=1-2, to=1-3]
\end{tikzcd}
\end{equation}
to define the functor
\begin{equation} \label{equation spectral quantization of twisted 1-morphisms}
    \bB(X,\alpha):= p_{2,*} \big(p_1^!(-)^\shear \otimes^! (\theta^X_{\Del,*}\boldsymbol{\alpha}^!\mathbf{exp})\big)^\unshear: \QA^!(\Del_{G_1}(C)) \to \QA^!(\Del_{G_2}(C)),
\end{equation}
representing the $1$-morphism of $\Cc$ defined by $(X,\alpha)$.

We finish the section studying the composition of the representation of $1$-morphisms.

\begin{lemma}
    Let $M_{12} = T^*_\alpha X \in \Mor_{\Cc}(G_1, G_2)$ and $M_{23} = T^*_\beta Y \in \Mor_{\Cc}(G_2, G_3)$ and consider $M_{13} = M_{23} \circ M_{12} \in \Mor_{\Cc}(G_1, G_3)$ given by $Z = X \times Y/G_2$. Then there is an equivalence of functors
    $$\bB(Z, \alpha \times \beta) \simeq \bB(Y, \beta) \circ \bB(X, \alpha): \QA^!(\Del_{G_1}(C)) \to \QA^!(\Del_{G_3}(C)).$$
\end{lemma}
\begin{proof}
     Recall that composition of functors $B(Y, \beta) \circ B(X,\alpha)$ is realized by the integration kernel 
    \begin{equation} \label{equation composed kernel}
        p_{13,*}\big(p_{12}^! \bB(X,\alpha) \otimes^! p_{23}^! \bB(Y, \beta)\big) \in \QA^!(\Del_{G_1}(C) \times \Del_{G_3}(C))
    \end{equation}
    where we have abused notation slightly to write $\bB(X,\alpha), \bB(Y, \beta)$ as their representing kernels, and $p_{ij}$ denotes the projection map on $\Del_{G_1} \times \Del_{G_2} \times \Del_{G_3}$ onto its $i$th and $j$th components. To compute this integration kernel, we write down the relative Deligne moduli stacks responsible for each tensor factor
    $$\begin{tikzcd}
	& {\Del_{G_1 \times G_2 \times G_3}^{X \times Y} \simeq \Del_{G_1 \times G_3}^Z} & \\
	{\Del_{G_1 \times G_2}^X} & {\Del_{G_1 \times G_3}} & {\Del_{G_2 \times G_3}^Y} \\
	{\Del_{G_1 \times G_2}} & {\Del_{G_1 \times G_2 \times G_3}} & {\Del_{G_2 \times G_3}}
	\arrow["{q_{12}}"', from=1-2, to=2-1]
	\arrow["{\theta^Z}"', from=1-2, to=2-2]
	\arrow["{q_{23}}", from=1-2, to=2-3]
	\arrow["{\theta^X}", from=2-1, to=3-1]
	\arrow["{\theta^Y}"', from=2-3, to=3-3]
	\arrow["{p_{13}}", from=3-2, to=2-2]
	\arrow["{p_{12}}"', from=3-2, to=3-1]
	\arrow["{p_{23}}", from=3-2, to=3-3]
\end{tikzcd}$$
where the equivalence on the top is induced by the equivalence $X \times Y/(G_1 \times G_2 \times G_3) \simeq Z/G_1 \times G_3$. Since the objects $\bB(X,\alpha), \bB(Y,\beta)$ are sheaves on $\Del_{G_1 \times G_2}^X, \Del_{G_2 \times G_3}^Y$ pushed forward along $\theta^X, \theta^Y$, respectively, we may rewrite the object of \eqref{equation composed kernel} equivalently as
\begin{equation}
    \theta^Z_*\big(( \boldsymbol{\alpha} \circ q_{12})^!\mathbf{exp} \,  \otimes^! \,  (\boldsymbol{\beta} \circ q_{23})^!\mathbf{exp}\big) \in \QA^!(\Del_{G_1 \times G_3}(C)).
\end{equation}
By the identification of Lemma \ref{lemma composition of twisted polarization}, we recognize that the preceding expression is exactly the object $\bB(Z,\alpha \times \beta)$, as we wanted to show.
\end{proof}

\subsubsection{2-morphisms} \label{sc 2-mor}

As we did in Section \ref{sc 1-mor}, we start considering the case of a trivial polarization ({\it i.e.} $\alpha = 0$) to ease the reading.

Suppose that we have a polarized morphism $\varphi: M_1 = T^*X_1 \to M_2 = T^*X_2$ of graded Hamiltonian $G$-actions, {\it i.e.} $\varphi$ restricts to a morphism $X_1 \to X_2$ of graded $G$-schemes with $X_i$ as in the statement of Corollary \ref{co conditions for [X/H]}. Then $\varphi$ induces a diagram of quasi-algebraic morphisms 
\begin{equation}
    \begin{tikzcd}
	{\Del_G^{X_1}(C)} && {\Del_G^{X_2}(C)} \\
	& {\Del_G(C)}
	\arrow["{\boldsymbol{\varphi}}", from=1-1, to=1-3]
	\arrow["{\theta^{X_1}}"', from=1-1, to=2-2]
	\arrow["{\theta^{X_2}}", from=1-3, to=2-2]
\end{tikzcd},
\end{equation}
that commuting with the structural projections onto $\P^1$. By adjunction, we have a morphism of quasi-algebraic sheaves on $\Del_G(C)$,
\begin{equation} \label{equation Bphi}
\bB(\varphi):\bB(X_1) = \theta_{*}^{X_1}\omega = \theta_{*}^{X_2}\boldsymbol\varphi_*\omega \simeq \theta_{*}^{X_2}(\boldsymbol\varphi_*\boldsymbol\varphi^!\omega) \to \theta_{*}^{X_2}\omega = \bB(X_2).
\end{equation}
Specializing to the case when $G = G_1 \times G_2$ is a product group, and regarding $\bB(X_1), \bB(X_2): \QA^!(\Del_{G_1}(C)) \to \QA^!(\Del_{G_2}(C))$ as functors on categories of quasi-algebraic sheaves, \eqref{equation Bphi} induces a natural transformation
\begin{equation}
    \begin{tikzcd}
	{\QA^!(\Del_{G_1}(C))} && {\QA^!(\Del_{G_2}(C))}
	\arrow[""{name=0, anchor=center, inner sep=0}, "{\bB(X_1)}", bend left, from=1-1, to=1-3]
	\arrow[""{name=1, anchor=center, inner sep=0}, "{\bB(X_2)}"', bend right, from=1-1, to=1-3]
	\arrow["{\bB(\varphi)}", Rightarrow, from=0, to=1]
\end{tikzcd}
\end{equation}
by acting on the integration kernels defining $\bB(X_1)$ and $\bB(X_2)$. 

The preceding construction can be generalized to the twisted polarized setting. Given a pair of twisted polarized Hamiltonian spaces $M_i = T^*_{\alpha_i}X_i$ for $i = 1,2$, and suppose we have a $G \times \Ggr$ equivariant Lagrangian correspondence $Z \to M_1 \times M_2$ inducing a 0-Lagrangian morphism
$$z: Z^\shear/G \longrightarrow M_1^\shear/G \times_{T^*[3]BG} M_2^\shear/G,$$
such that $z$ is equipped with a polarization, {\it i.e.} there is a diagram
    \begin{equation}
        \begin{tikzcd}
	{M_1} & Z & {M_2} \\
	{X_1} & A & {X_2}
	\arrow[from=1-1, to=2-1]
	\arrow[from=1-2, to=1-1]
	\arrow[from=1-2, to=1-3]
	\arrow["\pi", from=1-2, to=2-2]
	\arrow[from=1-3, to=2-3]
	\arrow["{p_1}"', from=2-2, to=2-1]
	\arrow["{p_2}", from=2-2, to=2-3]
\end{tikzcd}
    \end{equation}
    so that $\pi$ is a Lagrangian fibration and we have an identification $p_1^!\alpha_1 \simeq p_2^!\alpha_2$, and the morphism $p_1$ is \'etale.

In this situation, we have a diagram of quasi-algebraic stacks over $\P^1$,
\begin{equation}
    \begin{tikzcd}
	& {\Del_G^A(C)} & \\
	{\Del_G^{X_1}(C)} && {\Del_G^{X_2}(C)} \\
	& {\Del_G(C)}
	\arrow["a"', from=1-2, to=2-1]
	\arrow["b", from=1-2, to=2-3]
	\arrow["p"', from=2-1, to=3-2]
	\arrow["q", from=2-3, to=3-2]
\end{tikzcd}
\end{equation}
where $a$ is \'etale, which can be verified by checking its relative cotangent complex. By the standard yoga of correspondences, we may produce a morphism
\begin{equation}
    \begin{tikzcd}
	{\QA^!(\Del_{G_1}(C))} && {\QA^!(\Del_{G_2}(C))}
	\arrow[""{name=0, anchor=center, inner sep=0}, "{\bB(X_1,\alpha_1)}", bend left, from=1-1, to=1-3]
	\arrow[""{name=1, anchor=center, inner sep=0}, "{\bB(X_2,\alpha_2)}"', bend right, from=1-1, to=1-3]
	\arrow["{\bB(z)}", Rightarrow, from=0, to=1]
\end{tikzcd}
\end{equation}
via the following morphism of sheaves on $\Del_G(C)$:
\begin{equation}
    \bB(z): \bB(X_1, \alpha_1) \overset{\varepsilon}{\to} (p_*a_*)(a^! \boldsymbol{\alpha}_1)\simeq (q_*b_*)(b^!\boldsymbol{\alpha}_2) \to \bB(X_2,\alpha_2)
\end{equation}
where $\varepsilon$ is induced by the unit $\mathrm{id} \to a_*a^* \simeq a_* a^!$ of the adjunction $(a^* \simeq a^!, a_*)$ afforded by \'etaleness of $a$.


We address now the composition of $2$-morphisms.

\begin{lemma}
Given the $2$-morphisms $Z_{12} \in \Mor_{\Cc}(T^*_{\alpha_1} X_1, T^*_{\alpha_2} X_2)$ and $Z_{23} \in \Mor_{\Cc}(T^*_{\alpha_2} X_2, T^*_{\alpha_3} X_3)$ and consider their composition $Z_{13} = Z_{23} \circ Z_{12} \in \Mor_{\Cc}(T^*_{\alpha_1} X_1, T^*_{\alpha_3} X_3)$. There is an equivalence of natural transformations
    $$\bB(Z_{13}) \simeq \bB(Z_{13}) \circ \bB(Z_{12}): \bB(X_1,\alpha_1) \to \bB(X_3,\alpha_3).$$
\end{lemma}
\begin{proof}
The statement derives from the commutativity of \eqref{eq commuting diagram for composition}, and Cartesianity of the top part of the diagrams on each of the faces, which allow the use of base change theorems.   
\end{proof}

\subsubsection{Comparison with A-twist, or automorphic quantization}\label{subsubsection automorphic quantization}

Even though our constructions only concern the B-twist, we explain briefly the A-side as well for comparison of numerology. 

On the A-side, or automorphic side, one takes the data of the Moore--Tachikawa category and transgresses over \textit{1-Calabi--Yau objects} to consider the situation of \textit{0-shifted} quantization. More precisely, one picks a spin structure $K^{1/2}$ of the curve $C$, and performs an AKSZ transgression of the objects and morphisms of the Moore--Tachikawa category: for a $\Ggr$-equivariant 1-shifted Lagrangian $Z \to T^*[1]BG$, we apply the functor $\mathrm{Sect}_{K^{1/2}}(C, -)$ as in \cite{ginzburg&rozenblyum} to obtain a 0-Lagrangian morphism
\begin{equation}\label{equation Gaiotto Lagrangian}
    \mathrm{Sect}_{K^{1/2}}(C, Z) \longrightarrow \mathrm{Sect}_{K^{1/2}}(C, T^*[1]BG) \simeq \mathrm{Higgs}_G,
\end{equation}
where $\mathrm{Sect}_{K^{1/2}}(C, Z)$ are the \textit{Gaiotto Lagrangians} introduced in \textit{op. cit} and play a central role in the first author's analogous recent work on the A-side \cite{CHY}. 

Since $\mathrm{Higgs}_G \simeq T^*\mathrm{Bun}_G$ is polarized, the 0-shifted geometric quantization problem is relatively well-formulated: one seeks to lift Gaiotto's Lagrangians to $\mathcal{D}$-modules over $\mathrm{Bun}_G$ with \textit{microlocal support} given by \eqref{equation Gaiotto Lagrangian}. This latter is accomplished by BZSV's de Rham period sheaves \cite{BZSV} in the hyperspherical case, and more generally by recent work of Khan--Kinjo--Park--Safronov \cite{KKPS} when additional anomaly cancellation data is provided on the relevant Hamiltonian actions\footnote{More precisely, in \cite{KKPS} the authors have constructed perverse sheaves on $\mathrm{Bun}_G$ given Hamiltonian actions with anomaly cancellation data.}. 

On the B-side, or spectral side, one first \textit{shears} the data of the category $\mathcal{C}$ by the weight 2 $\Ggr$-action, and then one transgresses over \textit{2-Calabi--Yau objects} to consider the situation of \textit{1-shifted quantization}. More precisely, given a Hamiltonian action $Z \to T^*[1]BG$, we first shear by the $\Ggr$-action to obtain a 3-shifted Lagrangian
$$Z^\shear \longrightarrow (T^*[1]BG)^\shear \simeq \mathfrak{g}^*[2]/G.$$
and then we perform AKSZ transgression over the 2-Calabi--Yau object $C_{\Dol}$ to obtain a 1-shifted Lagrangian morphism
\begin{equation}\label{equation 1-shifted Gaiotto Lagrangian}
    \mathrm{Map}(C_{\Dol}, Z^\shear) \longrightarrow \mathrm{Map}(C_{\Dol}, \mathfrak{g}^*[2]/G) \simeq T^*[1]\Higgs_G.
\end{equation}
Since $T^*[1]\Higgs_G$ is polarized, the 1-shifted geometric quantization problem is well-formulated: one seeks to lift these 1-Lagrangians to ind-coherent sheaves on $\mathrm{Higgs}_G$ with \textit{coherent singular support} (in the sense of \cite{arinkin&gaitsgory}) given by \eqref{equation 1-shifted Gaiotto Lagrangian}, the 1-shifted analogue of Gaiotto's Lagrangians.

\subsection{The case of BZSV triples}
\label{sc BZSV}

This section, culminating in Corollaries \ref{co relation with BZSV} and \ref{co relation with Gaiotto}, is devoted to explain the relation of the constructions achieved in the preceding Section \ref{sc spectral quantization}, the spectral side of the relative Langlands program \cite{BZSV}, and Gaiotto's (BBB)-branes (denoted by $\mathcal{V}(G,C,M)$ in \cite{Gaiotto}). These serve as valuable data points which indicate that our quantization scheme interpolates known progress towards the construction of (BBB)-branes. 

\subsubsection{Hamiltonian $G$-spaces and BZSV triples}

A well-behaved collection of twisted polarized Hamiltonian actions have been singled out by the work of Ben-Zvi--Sakellarids--Venkatesh \cite{BZSV}, which we review briefly for the reader's convenience, although we remark that we do not restrict to \textit{hyperspherical actions} which appear to be exceptionally well-behaved under Langlands duality, hence their central role in \textit{op. cit}.

\begin{definition} \label{definition BZSV triple}
    Let $G$ be a reductive group. A \textit{BZSV triple} for $G$ is a triple $\mathbf{t} = (H, S, \rho)$, where
    \begin{itemize}
        \item $H$ is a reductive group with an inclusion $H \to G$, 
        \item $S$ is a finite dimensional symplectic representation of $H$ equipped with a commuting scaling $\GG_m$-action, and
        \item $\rho: \mathfrak{sl}_2 \to \mathfrak{g}$ is a Lie algebra homomorphism whose image commutes with the image of the Lie algebra $\mathrm{Im}(\mathfrak{h}) \subset \g$. We write $e,f,h$ for the standard basis of the source $\mathfrak{sl}_2$ and, when the context is clear, we abuse notation to think of $e,f,h$ as their $\rho$-image in $\mathfrak{g}$. 
    \end{itemize}
\end{definition}

\begin{remark}
In the first paragraph of Gaiotto--Witten's paper on $S$-duality of boundary conditions \cite{gaiotto&witten}, they have already identified such triples as the labeling data of boundary conditions for the SUSY Yang--Mills theory responsible for geometric Langlands.
\end{remark}

BZSV triples encode certain graded Hamiltonian $G$-actions via \textit{Whittaker induction}. Given a BZSV triple $(H, S, \rho)$, we consider the unipotent subgroup $U_\rho \subset G$ defined by $\rho$ by the following procedure: decomposing $\g$ via its $\rho(h)$-weights $\mathfrak{g} = \oplus_{j \in \ZZ} \, \mathfrak{g}_j$, we write
\begin{equation} \label{eq JM computation}
    \mathfrak{u}_\rho := \bigoplus_{k > 0} \, \mathfrak{g}_j \supseteq \mathfrak{u}_\rho^+ := \bigoplus_{k > 1} \, \mathfrak{g}_j
\end{equation}
which is a unipotent Lie subalgebra of the parabolic Lie subalgebra $\oplus_{k \geq 0} \, \mathfrak{g}_j$, which contains $\mathfrak{h}$. We write $U_\rho$, $U_\rho^+$ and $P_\rho = H U_\rho$ for the associated Lie subgroups in $G$ with Lie algebras $\mathfrak{u}_\rho$, $\mathfrak{u}_\rho^+$ and $\mathfrak{h} \oplus \mathfrak{u}_\rho$, as usual, and we regard $(\mathfrak{u}_\rho/\mathfrak{u}_\rho^+)_f$ as a Hamiltonian $P_\rho$-space where:
\begin{itemize}
\item the $H$-action is via the adjoint action, 
\item $U_\rho$ acts by translation via $U_\rho/U_\rho^+ \cong \mathfrak{u}_\rho/\mathfrak{u}_\rho^+$
\item the symplectic form is given by
\[
(x,y) \longmapsto \langle f, [x,y]\rangle, 
\]
and the moment map is given by
\[
(x,y) \longmapsto \langle f, [x,y]\rangle.
\]
\end{itemize}
Finally, the Hamiltonian $G$-space attached to the triple $(H, S, \rho)$ is obtained by \textit{symplectic induction} from $H U_\rho$ to $G$ by the formula
\begin{equation} \label{eq Whittaker induction}
\mathbf{t} = (H,S, \rho) \longmapsto M_{\mathbf{t}} := 
\bigg(S \times (\mathfrak{u}/\mathfrak{u}_+)_f \times^{HU}_{(\mathfrak{h}+\mathfrak{u})^*} \, T^*G\bigg),
\end{equation}
where $P_\rho$ acts diagonally on the product and $M_{\mathbf{t}}$ is equipped with a commuting $\GG_m$-action to be described below. 

By the theory of Slodowy slices, one can also rewrite $M_{\mathbf{t}}$ as a vector bundle over the homogeneous space $H \backslash G$ (although its symplectic nature becomes less obvious)
\begin{equation}
    M_{\textbf{t}} = \big[S \oplus (\mathfrak{h}^\perp \cap \mathrm{ker}(\mathrm{Ad}^*(e))\big] \times^H G.
\end{equation}
In this formulation, the commuting $\Ggr$ action on $M_{\mathbf{t}}$ is easier to describe: 
\begin{itemize}
\item $\Ggr$ acts by the weight 1 scaling action on $S$, 
\item $\Ggr$ acts on $\mathrm{ker}(\mathrm{Ad}^*(e))$ by weight $2+t$ on the direct summand $\mathrm{ker}(\mathrm{Ad}^*(e)) \cap \mathfrak{g}_t$, and
\item $\Ggr$ acts on $G$, after identifying $\mathrm{Lie}(\GG_m) = \mathrm{span}_\CC(h)$, by left multiplication by $\exp(\rho)$.
\end{itemize}

If $\rho = 0$ and $S = T^*V$ is a polarized representation, then the associated Whittaker induction produces a \textit{polarized} Hamiltonian action
\begin{equation}
    M_{\mathbf{t}} \simeq T^*(V \times^H G).
\end{equation}

\begin{remark}
As explained in Section 3.5.1 of \cite{BZSV}, the Hamiltonian $G$-space, $M_{\mathbf{t}}$, built out of the above Whittaker induction procedure will always satisfy 3 out of the 5 conditions of being a \textit{hyperspherical $G$-variety}; experts will observe that we are dropping the coisotropicity condition (which is a smallness condition on $M_{\mathbf{t}}$ relative to the $G$-action), and the condition that the stabilizer of a generic point in $M_{\mathbf{t}}$ be connected. For our purposes of constructing (BBB)-branes none of these relaxed conditions will present a problem, it is rather in the discussion of $S$-duality that we see consequences of leaving the hyperspherical regime. However, one has evidence (see Section 6 of \cite{toric periods} for instance) that the connected stabilizers assumption is not essential. 
\end{remark}

\subsubsection{Relation with BZSV's $L$-sheaves and Gaiotto's (BBB)-branes}
A canonical source of examples arises from BZSV triples with nontrivial $\rho$-component. Given $\mathbf{t} = (H,S = T^*V,\rho)$ a BZSV triple for $G$, we may consider the additive character
$$\psi_\rho: U_\rho \to U_\rho/[U_\rho,U_\rho] \overset{\simeq}{\to} \mathfrak{u}_\rho^{\mathrm{ab}} \overset{\rho(f)}{\to} \GG_a,$$
which induces the twisted polarization on $M_\mathbf{t}$ by
\begin{equation}
    \alpha_\mathbf{t}: X_\mathbf{t}/G = (V \times^{HU_\rho} G)/G \simeq V/HU_\rho\to BU_\rho \overset{B\psi_\rho}{\to} B\GG_a.
\end{equation}
Taking mapping stacks from Simpson shapes and gluing over $\P^1$, we obtain a diagram of quasi-algebraic stacks over $\P^1$
\begin{equation} \label{eq diagram for constructing B for t}
\begin{tikzcd}
	& {\Del_{HU_\rho}^V(C) \simeq \Del_G^{X_\mathbf{t}}}(C) & \\
	{\mathbf{O}(2)[-1]} && {\Del_G(C)}
	\arrow["{\boldsymbol{\alpha}_\mathbf{t}}"', from=1-2, to=2-1]
	\arrow["{\theta^{X_{\mathbf{t}}}}", from=1-2, to=2-3]
\end{tikzcd}
\end{equation}
where $\boldsymbol{\alpha}_\mathbf{t}$ is given in \eqref{eq def bf alpha}. The closed embedding $C_{\dR} \hookrightarrow C_{\Hod}$ induces, by pre-composition, the (algebraic) closed embedding 
\[
\jmath^X_{\dR} : \Loc_G^X(C) \hookrightarrow \Hodge_G^X(C)
\] 
Similarly, we recover from Section \ref{sc qa exponential sheaf} the morphisms
\[
\imath : \AA^1[-1] \hookrightarrow \O(2)[-1]
\]
given by restriction to the fibre over $1 \in \PP^1$. Hence, the quasi-algebrification of the above provides the following morphism within the quasi-algebraic category,  
\[
\mathbf{j}^X_{\dR} : \Loc_G^X(C)^{\qalg} \hookrightarrow \Del_G^X(C)
\]
and
\[
\mathbf{i} : \AA^1[-1] \hookrightarrow \OO(2)[-1],
\]
where we have dropped the superindex $\qalg$ from the source to ease the reading. Consider as well
\begin{equation}\label{eq diagram for constructing B_dR for t}
\begin{tikzcd}
	& {\Loc_{HU_\rho}^V(C) \simeq \Loc_G^{X_\mathbf{t}}}(C) & \\
	{\AA[-1]} && {\Loc_G(C)}
	\arrow["{\alpha_{\dR,\mathbf{t}}}"', from=1-2, to=2-1]
	\arrow["{\theta_{\dR}^{X_{\mathbf{t}}}}", from=1-2, to=2-3]
\end{tikzcd},
\end{equation}
where $\alpha_{\dR,\mathbf{t}}$ is obtained as in \eqref{eq def bf alpha}, composing the de Rahm orientation $[C_{\dR}]$ with the morphism obtained by post-composition under $\alpha_{\mathbf{t}}$, and $\theta^X_{\dR}$ is provided by post-composition with respect to $[X/G] \to BG$. 

Recall from Section \ref{sc qa exponential sheaf} the (algebraic) exponential sheaf $\exp_{\AA^1} \in \QC^!(\AA^1[-1])^{\shear}$ as constructed in \cite{BZSV}. 
With all of the above, one can now remind the following definition which is central in the work of {\it op. cit.}

\begin{definition}[Definition 11.6.4 of \cite{BZSV}]
Given a BZSV triple $\mathbf{t} = (H, S = T^*V, \rho)$, define the associated (unnormalized) L-sheaf as 
\[
\Ll_{\mathbf{t}} := (\theta_{\dR, *}^X\alpha_{\dR, \mathbf{t}}^!\exp_{\AA^1})^\unshear \in \QC^!(\Loc_G(C))
\]
\end{definition}

By construction, the diagram \eqref{eq diagram for constructing B for t} specializes to \eqref{eq diagram for constructing B_dR for t} in the sense that the following diagrams are Cartesian,
\[
\begin{tikzcd}
	{\Loc_G^{X_\mathbf{t}}(C)} & {\Del_G^{X_\mathbf{t}}(C)} \\
	{\AA^1[-1]} & {\OO(2)[-1]}
	\arrow[from=1-1, to=1-2, "\mathbf{j}^X_{\dR}"]
	\arrow[from=1-1, to=2-1, "\alpha_{\dR, \mathbf{t}}^{\qalg}"']
	\arrow[from=1-2, to=2-2, "\boldsymbol{\alpha}_{\mathbf{t}}"]
	\arrow[from=2-1, to=2-2, "\mathbf{i}"']
\end{tikzcd}, \quad \begin{tikzcd}
	{\Loc_G^{X_\mathbf{t}}(C)^{\qalg}} & {\Del_G^{X_\mathbf{t}}(C)} \\
	{\Loc_G(C)^{\qalg}} & {\Del_G(C)}
	\arrow[from=1-1, to=1-2, "\mathbf{j}^{X_{\mathbf{t}}}_{\dR}"]
	\arrow[from=1-1, to=2-1, "\theta_{\dR}^{X_{\mathbf{t}},\qalg}"']
	\arrow[from=1-2, to=2-2, "\theta^{X_{\mathbf{t}}}"]
	\arrow[from=2-1, to=2-2, "\mathbf{j}_{\dR}"']
\end{tikzcd}.
\]

The statement below follows immediately after base change under the preceding diagrams.

\begin{corollary} \label{co relation with BZSV}
    Let $\bB(X_{\mathbf{t}},\alpha_{\mathbf{t}}) \in \QA^!(\Del_G(C))$ be the quasi-algebraic sheaf associated to the (Hamiltonian action encoded by the) BZSV triple $\mathbf{t} = (H, S = T^*V, \rho)$ for $G$. Then its underlying ind-coherent sheaf over the de Rham moduli stack coincides with BZSV's (unnormalized) $L$-sheaf associated to the BZSV triple $\mathbf{t}$, {\it i.e. }
    $$\mathbf{j}_{\dR}^!\bB(X_\mathbf{t},\alpha_\mathbf{t}) \simeq \left ( \Ll_{\mathbf{t}}\right )^{\qalg}.$$
\end{corollary}
Similarly, we have a Dolbeault restriction to the (relative) Hitchin moduli stack
$$\mathbf{j}_{\Dol}^X: \Higgs_G^X(C) \to \Del_G^X(C)$$
along which one can restrict any quasi-algebraic sheaf on $\Del_G^X(C)$.
\begin{corollary}\label{co relation with Gaiotto}
    Let $\bB(X_\mathbf{t}, \alpha_{\mathbf{t}})$ be the quasi-algebraic sheaf associated to the (Hamiltonian action encoded by the) BZSV triple $\mathbf{t} = (H, S = T^*V,0)$ for $G$. Suppose that the Dirac--Higgs bundle on $H$ with coefficients in $V$ is concentrated in cohomological degree 1 when restricted to the stable moduli space. Then its underlying ind-coherent sheaf over the de Rham moduli stack coincides with Gaiotto's $\mathcal{V}(G,C,M_{\mathbf{t}})$.
\end{corollary}
\begin{proof}
    When $\mathbf{t} = (H, S = T^*V, 0)$ is a representation of cotangent type for a reductive subgroup $H \subset G$, Gaiotto proposes the definition 
    $$\mathcal{V}(G,C,M_{\mathbf{t}}) := \wedge^\bullet \mathbf{H}^1(C, \Ee_V \overset{\varphi}{\to} \Ee_VK_C)$$
    where $\Ee_V \overset{\varphi}{\to} \Ee_VK$ is the universal $H$-Higgs field with coefficients in $V$, regarded as a derived vector bundle on $\Higgs_H^{\mathrm{st}}$, which one then pushes forward to $\Higgs_G$. 

    In this case, the relative stack $\Higgs_G^X \simeq \Higgs_H^V$ is the total space of the Dirac--Higgs complex $\mathbf{H}(C, \Ee_V \overset{\varphi}{\to} \Ee_V K_C)$ over $\Higgs_H$. Thus, the pushforward of the space of distributions along $\Higgs_H^V \to \Higgs_H$ coincides with the relative symmetric algebra of the Dirac--Higgs complex itself. With the assumption that this complex has support in cohomological degree 1 when restricted to $\Higgs_H^{\mathrm{st}}$, we see that 
    $$\mathbf{j}_{\Dol}^{X, !}\bB(X_{\mathbf{t}}, \alpha_{\mathbf{t}})|_{\Higgs_H^{\mathrm{st}}} \simeq \mathrm{Sym}^\bullet \mathbf{H}(C, \Ee_V \overset{\varphi}{\to} \Ee_V K_C) \simeq \wedge^\bullet \mathbf{H}^1(C, \Ee_V \overset{\varphi}{\to} \Ee_VK_C)$$
    by the Koszul sign rule, as we wanted to confirm.
\end{proof}

\section{Towards hyperK\"ahler structures and (BBB)-branes}
\label{sc BBB-branes}

This last section is devoted to motivating our main result Theorem \ref{th representation of Cc} as an initial step towards the B-twist construction of boundary conditions for the Langlands TQFT. This construction is summarized in Conjecture \ref{cj B-twist}.

\subsection{Shifted (pre-)twistor structures}

Penrose's work on twistor theory \cite{penrose} leads to the construction of the so-called \emph{twistor space} associated with a smooth hyperK\"ahler manifold $M$. Subsequently, the work of Hitchin--Karlhede--Lindstr\"om--Ro\v{c}ek \cite{HKLR} provides a reconstruction of the hyperK\"ahler structure of $M$ from holomorphic data defined over its twistor space. Elaborating on \cite{HKLR}, Katzarkov--Pandit--Spaide provided in \cite{KPS} a tentative definition of a hyperK\"ahler structure in the context of derived geometry recently refined by Kryczka--Tannaka--Yau \cite{kryczka&tannaka&yau}.

The twistor space of a hyperK\"ahler manifold $M$ is a complex analytic manifold $\Tw(M)$ equipped with a structural morphism to $\PP^1$, a relative symplectic form on its fibres, and a real structure $\chi$ covering the antipodal map on $\PP^1$. Furthermore, the twistor space is equipped with a $C^\infty$ isomorphism $\Tw(M) \cong M \times \PP^1$, which allows one to recognize as {\it horizontal} those sections of the structural morphism $\Tw(M) \to \PP^1$ whose images are sent to a fixed slice $\{ m \} \times \PP^1$ under the latter.

The aforementioned work of Hitchin--Karlhede--Lindstrom--Ro\v{c}ek \cite{HKLR} allows us to reverse this process. Starting from an analytic variety $\Zz$ over $\PP^1$ equipped with a real structure $\chi$ as above, they propose an intrinsic definition of {\it horizontality} based on the idea that such sections could be deformed without obstruction.

\begin{definition} \label{def horizontal twistor section in analytic varieties} 
Given an analytic smooth fibre bundle $\tau: \Zz \to \PP^1$, we say that a $\chi$-equivariant section $\sigma: \PP^1 \to \Zz$ of $\tau $ is a \textit{horizontal twistor section} if the normal bundle of $\sigma$ is isomorphic to $N_\sigma \cong \Oo_{\PP^1}(1)^{\oplus d}$, with $d = (\dim \Zz - 1)$.
\end{definition}

Note that the above definition implies that $H^1(\PP^1, N_\sigma) = 0$, so, by Kodaira's theorem, $\sigma$ can be deformed unobstructedly along the tangent space $H^0(\PP^1, N_\sigma)$ of dimension $d$. This is one of the crucial steps of the proof of the following statement.

\begin{theorem}[Th. 3.3(ii) of \cite{HKLR}] \label{tm HKLR}
Let $\tau: \Zz \to \PP^1$ be an analytic fibre bundle over $\PP^1$ admitting a family of horizontal twistor sections. Suppose that there exists a relative symplectic structure on the fibres of $\tau$, and a compatible real structure $\chi : \Zz \to \Zz$ inducing the antipodal map on $\PP^1$. Then the parameter space of horizontal twistor sections is a hyperK\"ahler manifold with twistor space $\Zz$.
\end{theorem}

Based on the previous statement, Katzarkov--Pandit--Spaide \cite{KPS} defined an analogous (tentative) notion of twistor spaces in the framework of derived geometry, recently refined by Kryczka--Tannaka--Yau \cite{kryczka&tannaka&yau}. In their terms, an {\it $m$-shifted derived twistor family of hyperKähler type} is a morphism of derived analytic stacks $\tau : \Zz \to \PP^1$ together with the following data:
\begin{enumerate}

\item ({\it A real structure}:) A homotopy-coherent action of $\mathrm{Gal}(\CC/\RR)$ whose underlying data consists of an automorphism $\chi:\Zz \to \Zz$ covering the antipodal involution of $\PP^1$, a homotopy $\chi^2 \simeq \id_{\Zz}$ compatible with higher coherences.

\item ({\it A symplectic structure}:) A relative $m$-shifted symplectic structure $\Omega$ on the fibers of $\tau$, compatible with the real structure $\chi$. Its underlying $2$-form is a degree $m$ section
\[
\Omega \in \Gamma\!\left( \Zz, \wedge^2 \LL^{\an}_{\Zz/\PP^1} \otimes\tau^*\Oo_{\PP^1}(2) \right),
\]
where $\LL^{\an}_{\Zz/\PP^1}$ denotes the relative cotangent complex. The non-degeneracy condition requires that $\Omega$ induces a quasi-isomorphism
\[
(\LL^{\an}_{\Zz/\PP^1})^\vee
\simeq
\LL^{\an}_{\Zz/\PP^1}[m] \otimes \tau^*\Oo_{\PP^1}(2),
\]
ensuring that every fiber carries an $m$-shifted symplectic structure.

\item ({\it A collection of horizontal twistor sections}:) A (possibly disconnected) union of connected components $\Yy \subseteq \Maps_{\mathbb{P}^1}\left (\PP^1,\Zz \right )^{h \mathrm{Gal}(\CC/\RR)}$ of the homotopy fixed points of the induced action of $\mathrm{Gal}(\CC/\RR)$ on the derived mapping stack of analytic sections of $\tau$, such that Zariski locally there exists an open substack $\Uu \subset \Zz$ whose classical truncation admits a good moduli space $:t_0(\Uu) \longrightarrow \Tw(M)$ for some underlying hyperKähler manifold $M$, and for which the canonical evaluation morphism
\begin{equation} \label{eq derived stack trivialization of twistor}
\Phi:\Yy \times \PP^1 \xrightarrow{\cong} \Uu
\end{equation}
is a $C^\infty$-equivalence compatible with the identification $\Tw(M)\cong M\times\PP^1$.

We remind that this condition involving the good moduli space was considered in \cite{franco&hanson, hanson_1} for the stacks appearing in non-abelian Hodge theory. 
\end{enumerate}

\begin{remark}
In \cite{kryczka&tannaka&yau} the authors adopt a weaker version of \eqref{eq derived stack trivialization of twistor} after passing to the classical truncation of $\Uu$. Here we present the one proposed in \cite{KPS}.
\end{remark}

When $\tau : \Zz \to \PP^1$ is only equipped with (1) and (2) it defines a {\it $m$-shifted derived pre-twistor family of hyperKähler type} in the terminology of \cite{kryczka&tannaka&yau}. In {\it loc. cit.} it is obtained by a shifted symplectic enhancement of the Riemann--Hilbert correspondence, which allows one to equip the (analytification) of the Deligne moduli stack with the structure of ($0$-shifted) derived pre-twistor family of hyperKähler type (recall the real structure $\chi^{\RR}$ constructed in Remark \ref{rm real form on Deligne}).

We would now like to focus on point (3), the \emph{collection of horizontal twistor sections}, whose description is still incomplete in the literature even if some steps were originally given in \cite{franco&hanson, hanson_1}. Our interest comes from the fact that horizontal twistor sections play a central role in the construction of the categories of (BBB)-branes, as we shall address next.

\subsection{Quasi-algebraic (BBB)-branes} 

(BBB)-branes were first introduced in the seminal work of Kapustin--Witten \cite{kapustin&witten}, appearing as objects which are holomorphic with respect to all complex structures of the Hitchin moduli space. These objects play an important role in the physical formulation of mirror symmetry underlying the Geometric Langlands correspondence. From a mathematical perspective, (BBB)-branes may be realized by hyperholomorphic bundles, which lift to holomorphic bundles on the twistor space by means of the Kaledin--Verbitsky correspondence \cite{kaledin&verbitsky}. The second named author, along with Hanson, constructed a dg-category \cite{franco&hanson} containing in its heart the (BBB)-branes provided by lifting the ones originally described over the Deligne moduli space. 

In this section, we first extend the notion of $m$-shifted derived pre-twistor family of hyperKähler type to the quasi-algebraic framework. Then, we propose a candidate for the appropriate collection of horizontal twistor sections, and, making use of the latter, a refinement of the notion of (BBB)-branes previously considered in \cite{franco&hanson, hanson_1}. Our motivation is to describe how we plan, in subsequent work, to extend the functor previously constructed in Section 4 to a more refined one landing in an eventual {\it category of (BBB)-branes}, in agreement with the program initiated by Kapustin and Witten for the representation of boundary conditions. 

Following \cite{KPS, kryczka&tannaka&yau}, we introduce the following notion in the quasi-algebraic setting.

\begin{definition}
A {\it quasi-algebraic $m$-shifted derived pre-twistor family of hyperKähler type} is a morphism of quasi-algebraic stacks $\tau : Z \to \P^1$ together with the following data:
\begin{enumerate}
\item ({\it A real structure}:) A homotopy-coherent action of $\mathrm{Gal}(\CC/\RR)$ covering the antipodal involution $\chi_\P^1 : \P^1 \to \P^1$.

\item ({\it A symplectic structure}:) A $\mathrm{Gal}(\CC/\RR)$-equivariant degree $m$ section
\[
\Omega \in \Gamma\!\left( Z, \wedge^2 \LL_{Z/\P^1} \otimes\tau^*\Oo_{\P^1}(2) \right),
\]
inducing a quasi-isomorphism $(\LL_{Z/\P^1})^\vee
\simeq \LL_{Z/\P^1}[m] \otimes \tau^*\Oo_{\P^1}(2)$.
\end{enumerate}
\end{definition}

Our proposal for the collection of horizontal twistor sections is essentially different from that envisaged by \cite{KPS, kryczka&tannaka&yau} and consists of taking the lifts of horizontal twistor sections on the good moduli space. This is a considerably modest class of horizontal sections to consider, although we do not expect our constructions to change meaningfully if a broader class of horizontal sections were to be proposed.

Consequently, we are required to work with \textit{split} quasi-algebraic stacks, whose classical truncation is equipped with a hyperK\"ahler good moduli space. 

Given a smooth quasi-algebraic scheme endowed with a twistor structure, note that Definition \ref{def horizontal twistor section in analytic varieties} extends automatically to quasi-algebraic schemes. We extend this notion further.

\begin{definition} \label{def lifted horizontal lines}
Let $\tau : Z \to \P^1$ be a quasi-algebraic $0$-shifted derived pre-twistor family of hyperKähler type such that $Z$ is split, {\it i.e. } equipped with a retraction $Z \to t_{0}(Z)$ of the inclusion of its classical truncation, and suppose that there exists an open dense subset $U \subset t_0(Z)$ equipped with a hyperK\"ahler good moduli space $\zeta : U \to \Mm$. Given a horizontal twistor section of the latter, $\sigma : \P^1 \to \Mm$, we define its {\it lift} to be the quasi-algebraic morphism
\[
\wt \sigma : \P^1_\sigma := \P^1 \times_{\Mm} U \to U,
\]
that projects to the second factor. In this case we say that $\wt \sigma$ is a {\it lifted horizontal twistor section}.
\end{definition}

Of course, the primary intended use case of these definitions is the class of quasi-algebraic Deligne moduli stacks with varying structure groups, where the open dense subset $U$ consists of the moduli stack of stable objects on which the nonabelian Hodge correspondence holds. The hyperK\"ahler good moduli space $\Mm$ is then the Deligne moduli space of stable objects. 

\begin{remark}
Observe that, by construction, the lifted horizontal twistor sections $\wt \sigma$ are unobstructed.
\end{remark}

By Theorem \ref{tm HKLR}, the hyperK\"ahler variety $M = \Mm|_{\{0 \} }$ has as its twistor space $\Mm$, and the collection of lifted horizontal twistor sections is parametrized by $M$ itself. Indeed, the family of horizontal twistor sections is {\it indexed} by $M_{\discrete}$, where the latter denotes the set of points of $M$ equipped with the discrete topology. Observe that the set of all horizontal twistor sections provide naturally the existence of a quasi-algebraic morphism,
\[
\Sigma : M_{\discrete} \times \P^1 \to \Mm, 
\]
whose restriction to each slice associated to $m \in M$ amounts to the horizontal twistor section $\sigma :  \P^1 \to \Mm$ indexed by $m$. We then consider the fibre product
\[
\P^1_{\Sigma} := \left ( M_{\discrete} \times \P^1 \right ) \times_{\Mm} U,
\]
consisting on the disjoint union of $\P^1_{\sigma}$ for all $\sigma$ parametrized by $M_{\discrete}$. Consider the natural projections onto each of its terms 
\[
\varphi : \P^1_{\Sigma} \to M_{\discrete}, \quad \beta : \P^1_{\Sigma} \to \P^1 \quad \text{and} \quad
\wt \Sigma : \P^1_{\Sigma} \to Z.
\]
One can think of $\P^1_{\Sigma}$ as a foliation of (a dense subset of) $Z$ whose leaves are the lifted horizontal twistor sections. The idea, inspired from the twistor correspondence (see for instance \cite{kaledin&verbitsky}), is to consider complexes $\Ee \in \QA^!(Z)$ whose pull-back under the two (quasi-algebraic) morphisms coincide. 

Importantly, the description of the relative tangent bundle in Remark \ref{rm Dirac for ad} suggests the following mild but conceptually crucial modification. 

Consider the complex with $0$ differential and having the degree $k$ line bundle in cohomological position $-k$,  
\[
\H = \bigoplus_k \Oo_{\P^1}(k)[k]  
\]
and define the functor
\[
\Xi : \QA^!(Z) \to \QA^!(Z), \quad \Ee \mapsto H^0(\Ee \otimes \tau^*\H) \cong \bigoplus_k H^k(\Ee)[-k].
\]
Observe that $\Xi(\Ee)$ is a complex supported in cohomological degree $0$, and it contains a copy of the cohomology of the original complex $\Ee$, degree-wise twisted by an appropriate line bundle.

\begin{definition}\label{definition ideal twistor category}
Consider the quasi-algebraic $0$-shifted derived pre-twistor family of hyperKähler type $\tau : Z \to \Mm$ with $Z$ split. Recall that $M_{\discrete}$ indexes the collection of lifted horizontal twistor sections, and consider the functors described above. Consider $k \in \frac{1}{2}\ZZ$ a half-integer. The {\it dg-category of hyperholomorphic quasi-algebraic sheaves with helicity $k$} on $Z$, denoted $\BBB^!_k(Z)$, is defined as the homotopy limit of the following solid diagram
    $$
    \begin{tikzcd}
	{\BBB^!_k(Z)} & & & & 
    {\QA^!(Z)} \\
	{\QA^!(M_{\discrete})} & & & & {\QA^!(\P^1_{\Sigma})}
	\arrow[dashed, from=1-1, to=1-5]
	\arrow[dashed, from=1-1, to=2-1]
	\arrow["{\Sigma^! \Xi (\cdot)}", from=1-5, to=2-5]
	\arrow["{\varphi^!(\cdot) \otimes^! \beta^!\Oo_{\P^1}(2k-2)}"', from=2-1, to=2-5]
\end{tikzcd}
    $$
    in dg-categories. We further abbreviate $\BBB^!_0(Z)$ by $\BBB^!(Z)$. 

As $\QA^!(M_{\discrete}) \cong \prod_{M_{\discrete}}\QA^!(\mathrm{pt})$ and $\QA^!(\P^1_{\Sigma}) \cong \prod_{M_{\discrete}}\QA^!(\P^1_\sigma)$, a hyperholomorphic sheaf is then a collection of data $(\Ee, \{\gamma_\sigma\}, \{V_\sigma\})$ where $\Ee$ is a quasi-algebraic quasi-algebraic sheaf on $Z$, along with \textit{horizontal trivialization data}: a complex of vector spaces $V_\sigma \in \QA^!(\mathrm{pt}) = \QA(\mathrm{pt})$ together with an isomorphism,
\[
\gamma_\sigma: \wt \sigma^!\Ee \stackrel{\simeq}{\to} p^!V_\sigma \otimes^! \Oo_{\P^1}(2k-2),
\]
for every lifted horizontal twistor section $\wt \sigma: \P^1 \to Z$ parametrized by $M_{\discrete}$.
\end{definition}

\begin{remark}
The above construction should be viewed as the derived analogue of the classical definition of $(BBB)$-branes which we recover by taking a heart of $\BBB^!_k(Z)$, for an appropriate choice of $k$.
\end{remark}

\begin{remark}
The functor $(\cdot) \otimes^! \tau^!\Oo_{\P^1}(2k-2)$ provides a one-to-one correspondence between objects of $\BBB^!_k(Z)$ and those of $\BBB^!_0(Z)$.
\end{remark}

Following the work of \cite{kryczka&tannaka&yau} the Deligne moduli stack can be equipped with a quasi-algebraic $0$-shifted derived pre-twistor family of hyperKähler type. Being $\Del_G(C)$ split by Remark \ref{rm Del is split} and since the open substack $\Del_G(C)^{\st}$ is naturally equipped with a good moduli structure $\zeta: \Del_G(C)^{\st} \to \Mm_{\Del}(C,G)^{\st}$, one can consider the set of horizontal twistor sections for $\Mm_{\Del}(C,G)^{\st}$ to obtain the corresponding set of lifted horizontal twistor sections. With that, one can build the category of (BBB)-branes, by means of the solid diagram in dg-categories,
\begin{equation} \label{eq solid diagram for Del}
    \begin{tikzcd}
	{\BBB^!_k(\Del_G(C))} & & & & 
    {\QA^!(\Del_G(C))} \\
	{\QA^!(\Mm_{\Del}(C,G)^{\st}_{\discrete})} & & & & {\QA^!(\P^1_{\Sigma})}
	\arrow[dashed, from=1-1, to=1-5]
	\arrow[dashed, from=1-1, to=2-1]
	\arrow["{\Sigma^! \Xi (\cdot)}", from=1-5, to=2-5]
	\arrow["{\varphi^!(\cdot) \otimes^! \beta^!\Oo_{\P^1}(-2)}"', from=2-1, to=2-5]
\end{tikzcd}.
\end{equation}

\begin{example}
Consider the relative tangent complex $\TT_{\Del_{\GG_m}/\P^1}$ and its twist by $\tau^! \Oo_{\P^1}(-1)$,  
\[
\DD_{\ad} = \TT_{\Del_{\GG_m}/\P^1} \otimes^! \tau^! \Oo_{\P^1}(-1),
\]
which amounts to the Dirac--Higgs bundle for the trivial reprensentation.

As an immediate consequence of Remark \ref{rm Dirac for ad}, one can show that both $\TT_{\Del_{\GG_m}/\P^1}$ and $\DD_{\ad}$ are hyperholomorphic quasi-algebraic sheaves, with helicities $1$ and $0$, respectively, {\it i.e.}
\[
\TT_{\Del_{\GG_m}/\P^1} \in \BBB^!_1(\Del_{\GG_m}), \quad \text{and} \quad \DD_{\ad} \in \BBB^!_0(\Del_{\GG_m}).
\]
\end{example}

\subsection{Next steps} \label{subsect next steps}

In this last section we outline the next steps of our program by encapsulating them in a conjecture. 

For technical reasons, we consider a subcategory $\Cc_{\mathrm{qaff}} \subset \Cc$ containig the same objects, but whose 1-morphisms consist only of \textit{quasi-affine} Hamiltonian actions. 
\begin{conjecture} \label{cj B-twist}
Let $G$ be a reductive group, and $C$ a smooth projective curve. There exists a collection of horizontal twistor sections on $\Del_G(C)$, containing those of Definition \ref{def lifted horizontal lines}, equipping $\Del_G(C)$ with the structure of a quasi-algebraic twistor family of hyperKähler type. With respect to this twistor structure, there exists a representation of the polarized \textit{quasi-affine} Moore--Tachikawa category 
\[
B : \Cc_{\mathrm{qaff}} \to \mathrm{dgCat}
\]
sending the objects of $G \in \Cc_{\mathrm{qaff}}$ to 
\[
B(G) := \BBB^!(\Del_G(C)), 
\]
lifting the functor $\bB$ of Theorem \ref{th representation of Cc} along the functor $\BBB^!(\Del_G(C)) \to \QA^!(\Del_G(C))$ which forgets horizontal trivializations.
\end{conjecture}
\begin{remark}
    The restriction to $\Cc_{\mathrm{qaff}} \subset \Cc$ appears to be necessary. In examples, we can observe that if the Hamiltonian space contains nontrivial rational curves, it would not be possible in general to equip the resulting quasi-algebraic structures with horizontal trivialization data. There are several ways to address this problem, but we will leave this discussion to future work.
\end{remark}
Observe that the functor $\bB : \Cc \to \mathrm{dgCat}$ constructed in Theorem \ref{th representation of Cc} amounts to the top-right corner of the diagram \eqref{eq solid diagram for Del}. In subsequent work, we equip $\bB|_{\Cc_{\mathrm{qaff}}}$ with horizontal trivializations to obtain the desired functor of the preceding conjecture.

\end{document}